\documentclass[11pt]{amsart}
\usepackage[utf8]{inputenc}
\usepackage{kantlipsum} % for text filler
\calclayout
\newif\ifdraft
\draftfalse
\usepackage{amsmath,amsfonts,amsthm,mathrsfs}
\usepackage{amssymb}

\usepackage[unicode,bookmarks]{hyperref}
\usepackage[usenames,dvipsnames]{xcolor}
\hypersetup{colorlinks=true,citecolor=NavyBlue,linkcolor=Brown,urlcolor=Orange}

\usepackage[alphabetic,initials]{amsrefs}

\usepackage{enumitem}
\newenvironment{renumerate}{%
	\begin{enumerate}[label=(\roman{*}), ref=(\roman{*})]
}{%
	\end{enumerate}%
}
\newenvironment{nenumerate}{%
	\begin{enumerate}[label=(\arabic{*}), ref=(\arabic{*})]
}{%
	\end{enumerate}%
}

\usepackage{chngcntr}

\ifdraft
\usepackage[notcite,notref,color]{showkeys}

\definecolor{labelkey}{gray}{0.5}
\fi

\usepackage{tikz}

\usepackage{tikz-cd}
\usetikzlibrary{matrix,arrows}
\newlength{\myarrowsize} 

\pgfarrowsdeclare{cmto}{cmto}{
	\pgfsetdash{}{0pt} 
	\pgfsetbeveljoin 
	\pgfsetroundcap 
	\setlength{\myarrowsize}{0.6pt}
	\addtolength{\myarrowsize}{.5\pgflinewidth}
	\pgfarrowsleftextend{-4\myarrowsize-.5\pgflinewidth} 
	\pgfarrowsrightextend{.8\pgflinewidth}
}{
	\setlength{\myarrowsize}{0.6pt} 
  	\addtolength{\myarrowsize}{.5\pgflinewidth}  
	\pgfsetlinewidth{0.5\pgflinewidth}
	\pgfsetroundjoin
	\pgfpathmoveto{\pgfpoint{1.5\pgflinewidth}{0}}
	\pgfpatharc{-109}{-170}{4\myarrowsize}
	\pgfpatharc{10}{189}{0.58\pgflinewidth and 0.2\pgflinewidth}
	\pgfpatharc{-170}{-115}{4\myarrowsize+\pgflinewidth}
	\pgfpathclose
	\pgfusepathqfillstroke
	\pgfpathmoveto{\pgfpoint{1.5\pgflinewidth}{0}}
	\pgfpatharc{109}{170}{4\myarrowsize}
	\pgfpatharc{-10}{-189}{0.58\pgflinewidth and 0.2\pgflinewidth}
	\pgfpatharc{170}{115}{4\myarrowsize+\pgflinewidth}
	\pgfpathclose
	\pgfusepathqfillstroke
	\pgfsetlinewidth{2\pgflinewidth}
}

\pgfarrowsdeclare{cmonto}{cmonto}{
	\pgfsetdash{}{0pt} 
	\pgfsetbeveljoin 
	\pgfsetroundcap 
	\setlength{\myarrowsize}{0.6pt}
	\addtolength{\myarrowsize}{.5\pgflinewidth}
	\pgfarrowsleftextend{-4\myarrowsize-.5\pgflinewidth} 
	\pgfarrowsrightextend{.8\pgflinewidth}
}{
	\setlength{\myarrowsize}{0.6pt} 
  	\addtolength{\myarrowsize}{.5\pgflinewidth}  
	\pgfsetlinewidth{0.5\pgflinewidth}
	\pgfsetroundjoin
	\pgfpathmoveto{\pgfpoint{1.5\pgflinewidth}{0}}
	\pgfpatharc{-109}{-170}{4\myarrowsize}
	\pgfpatharc{10}{189}{0.58\pgflinewidth and 0.2\pgflinewidth}
	\pgfpatharc{-170}{-115}{4\myarrowsize+\pgflinewidth}
	\pgfpathclose
	\pgfusepathqfillstroke
	\pgfpathmoveto{\pgfpoint{1.5\pgflinewidth}{0}}
	\pgfpatharc{109}{170}{4\myarrowsize}
	\pgfpatharc{-10}{-189}{0.58\pgflinewidth and 0.2\pgflinewidth}
	\pgfpatharc{170}{115}{4\myarrowsize+\pgflinewidth}
	\pgfpathclose
	\pgfusepathqfillstroke
	\pgfpathmoveto{\pgfpoint{1.5\pgflinewidth-0.3em}{0}}
	\pgfpatharc{-109}{-170}{4\myarrowsize}
	\pgfpatharc{10}{189}{0.58\pgflinewidth and 0.2\pgflinewidth}
	\pgfpatharc{-170}{-115}{4\myarrowsize+\pgflinewidth}
	\pgfpathclose
	\pgfusepathqfillstroke
	\pgfpathmoveto{\pgfpoint{1.5\pgflinewidth-0.3em}{0}}
	\pgfpatharc{109}{170}{4\myarrowsize}
	\pgfpatharc{-10}{-189}{0.58\pgflinewidth and 0.2\pgflinewidth}
	\pgfpatharc{170}{115}{4\myarrowsize+\pgflinewidth}
	\pgfpathclose
	\pgfusepathqfillstroke
	\pgfsetlinewidth{2\pgflinewidth}
}

\pgfarrowsdeclare{cmhook}{cmhook}{
	\pgfsetdash{}{0pt} 
	\pgfsetbeveljoin 
	\pgfsetroundcap 
	\setlength{\myarrowsize}{0.6pt}
	\addtolength{\myarrowsize}{.5\pgflinewidth}
	\pgfarrowsleftextend{-4\myarrowsize-.5\pgflinewidth} 
	\pgfarrowsrightextend{.8\pgflinewidth}
}{
	\setlength{\myarrowsize}{0.6pt} 
  	\addtolength{\myarrowsize}{.5\pgflinewidth}  
 	\pgfsetdash{}{0pt}
	\pgfsetroundcap
	\pgfpathmoveto{\pgfqpoint{0pt}{-4.667\pgflinewidth}}
	\pgfpathcurveto
    {\pgfqpoint{4\pgflinewidth}{-4.667\pgflinewidth}}
    {\pgfqpoint{4\pgflinewidth}{0pt}}
    {\pgfpointorigin}
	\pgfusepathqstroke
}

\newenvironment{diagram*}[2]{%
\[%
\begin{tikzpicture}[>=cmto,baseline=(current bounding box.center),%
	to/.style={->,font=\scriptsize,cap=round},%
	into/.style={cmhook->,font=\scriptsize,cap=round},%
	onto/.style={-cmonto,font=\scriptsize,cap=round},%
	math/.style={matrix of math nodes, row sep=#2, column sep=#1,%
		text height=1.5ex, text depth=0.25ex}]%
}{%
\end{tikzpicture}%
\]% 
\ignorespacesafterend%
}

\newcommand{\derR}{\mathbf{R}}

\newcommand{\cohH}{\mathcal{H}}

\newcommand{\ZZ}{\mathbb{Z}}
\newcommand{\QQ}{\mathbb{Q}}

\newcommand{\CC}{\mathbb{C}}
\newcommand{\HH}{\mathbb{H}}
\newcommand{\PP}{\mathbb{P}}

\DeclareMathOperator{\im}{im}

\DeclareMathOperator{\coker}{coker}

\DeclareMathOperator{\Supp}{Supp}

\DeclareMathOperator{\Pic}{Pic}

\newcommand{\shf}[1]{\mathscr{#1}}

\newcommand{\restr}[1]{\big\vert_{#1}}

\def\overbar#1#2#3{{%
	\setbox0=\hbox{\displaystyle{#1}}%
	\dimen0=\wd0
	\advance\dimen0 by -#2 
	\vbox {\nointerlineskip \moveright #3 \vbox{\hrule height 0.3pt width \dimen0}%
		\nointerlineskip \vskip 1.5pt \box0}%
}}

\newcommand{\into}{\hookrightarrow}

\makeatletter
\let\@@seccntformat\@seccntformat
\renewcommand*{\@seccntformat}[1]{%
  \expandafter\ifx\csname @seccntformat@#1\endcsname\relax
    \expandafter\@@seccntformat
  \else
    \expandafter
      \csname @seccntformat@#1\expandafter\endcsname
  \fi
    {#1}%
}
\newcommand*{\@seccntformat@subsection}[1]{%
  \textbf{\csname the#1\endcsname.}
}
\makeatother

\makeatletter
\let\@paragraph\paragraph
\renewcommand*{\paragraph}[1]{%
	\vspace{0.3\baselineskip}%
	\@paragraph{\textit{#1}}%
}
\makeatother

\counterwithin{equation}{subsection}
\counterwithout{subsection}{section}
\counterwithin{figure}{subsection}

\newtheorem{theorem}[equation]{Theorem}
\newtheorem*{theorem*}{Theorem}
\newtheorem{lemma}[equation]{Lemma}
\newtheorem*{lemma*}{Lemma}
\newtheorem{corollary}[equation]{Corollary}
\newtheorem{proposition}[equation]{Proposition}
\newtheorem*{proposition*}{Proposition}

\theoremstyle{definition}
\newtheorem{definition}[equation]{Definition}
\newtheorem*{definition*}{Definition}
\newtheorem{remark}[equation]{Remark}

\newtheorem{example}[equation]{Example}
\newtheorem*{example*}{Example}
\newtheorem*{problem*}{Problem}

\newtheorem*{notation}{Notation}
\theoremstyle{plain}

\newcommand{\theoremref}[1]{\hyperref[#1]{Theorem~\ref*{#1}}}
\newcommand{\lemmaref}[1]{\hyperref[#1]{Lemma~\ref*{#1}}}
\newcommand{\definitionref}[1]{\hyperref[#1]{Definition~\ref*{#1}}}
\newcommand{\propositionref}[1]{\hyperref[#1]{Proposition~\ref*{#1}}}
\newcommand{\conjectureref}[1]{\hyperref[#1]{Conjecture~\ref*{#1}}}
\newcommand{\corollaryref}[1]{\hyperref[#1]{Corollary~\ref*{#1}}}
\newcommand{\exampleref}[1]{\hyperref[#1]{Example~\ref*{#1}}}

\makeatletter
\let\old@caption\caption
\renewcommand*{\caption}[1]{%
	\setcounter{figure}{\value{equation}}%
	\stepcounter{equation}%
	\old@caption{#1}\relax%
}
\makeatother

\newcounter{intro}

\newtheorem{intro-conjecture}[intro]{Conjecture}
\newtheorem{intro-corollary}[intro]{Corollary}
\newtheorem{intro-theorem}[intro]{Theorem}

\newcommand{\parref}[1]{\hyperref[#1]{\S\ref*{#1}}}

\makeatletter
\newcommand*\if@single[3]{%
  \setbox0\hbox{${\mathaccent"0362{#1}}^H$}%
  \setbox2\hbox{${\mathaccent"0362{\kern0pt#1}}^H$}%
  \ifdim\ht0=\ht2 #3\else #2\fi
  }
\newcommand*\rel@kern[1]{\kern#1\dimexpr\macc@kerna}
\newcommand*\widebar[1]{\@ifnextchar^{{\wide@bar{#1}{0}}}{\wide@bar{#1}{1}}}
\newcommand*\wide@bar[2]{\if@single{#1}{\wide@bar@{#1}{#2}{1}}{\wide@bar@{#1}{#2}{2}}}
\newcommand*\wide@bar@[3]{%
  \begingroup
  \def\mathaccent##1##2{%
    \if#32 \let\macc@nucleus\first@char \fi
    \setbox\z@\hbox{$\macc@style{\macc@nucleus}_{}$}%
    \setbox\tw@\hbox{$\macc@style{\macc@nucleus}{}_{}$}%
    \dimen@\wd\tw@
    \advance\dimen@-\wd\z@
    \divide\dimen@ 3
    \@tempdima\wd\tw@
    \advance\@tempdima-\scriptspace
    \divide\@tempdima 10
    \advance\dimen@-\@tempdima
    \ifdim\dimen@>\z@ \dimen@0pt\fi
    \rel@kern{0.6}\kern-\dimen@
    \if#31
      \overline{\rel@kern{-0.6}\kern\dimen@\macc@nucleus\rel@kern{0.4}\kern\dimen@}%
      \advance\dimen@0.4\dimexpr\macc@kerna
      \let\final@kern#2%
      \ifdim\dimen@<\z@ \let\final@kern1\fi
      \if\final@kern1 \kern-\dimen@\fi
    \else
      \overline{\rel@kern{-0.6}\kern\dimen@#1}%
    \fi
  }%
  \macc@depth\@ne
  \let\math@bgroup\@empty \let\math@egroup\macc@set@skewchar
  \mathsurround\z@ \frozen@everymath{\mathgroup\macc@group\relax}%
  \macc@set@skewchar\relax
  \let\mathaccentV\macc@nested@a
  \if#31
    \macc@nested@a\relax111{#1}%
  \else
    \def\gobble@till@marker##1\endmarker{}%
    \futurelet\first@char\gobble@till@marker#1\endmarker
    \ifcat\noexpand\first@char A\else
      \def\first@char{}%
    \fi
    \macc@nested@a\relax111{\first@char}%
  \fi
  \endgroup
}
\makeatother

\newcommand{\pp}{\mathfrak{p}}
\DeclareMathOperator{\Center}{center}
\DeclareMathOperator{\discrep}{\mathnormal{a}}
\DeclareMathOperator{\Cart}{Cart}
\DeclareMathOperator{\Weil}{Weil}
\DeclareMathOperator{\Divi}{Div}
\DeclareMathOperator{\Perv}{Perv}

\begin{document}

\vspace{\baselineskip}

\title{On the nonnegativity of stringy Hodge numbers}

\author[Sebasti\'an~Olano]{Sebasti\'an~Olano}
\address{Department of Mathematics, Northwestern University, 
2033 Sheridan Road, Evanston, IL
60208, USA} \email{{\tt seolano@math.northwestern.edu}}

\thanks{}

%\subjclass[2010]{14J17, 14F17, 32S25, 32S35}

\begin{abstract} We study the nonnegativity of stringy Hodge numbers of a projective variety with Gorenstein canonical singularities, which was conjectured by Batyrev. We prove that the $(p,1)$-stringy Hodge numbers are nonnegative, and for threefolds we obtain new results about the stringy Hodge diamond, which hold even when the stringy $E$-function is not a polynomial. We also use the Decomposition Theorem and mixed Hodge theory to prove Batyrev's conjecture for a class of fourfolds. 
\end{abstract}

\maketitle

\section{Introduction}

The purpose of this paper is to give some positive results towards a conjecture of Batyrev about the nonnegativity of stringy Hodge numbers. All varieties considered are over the field of complex numbers.\\

The stringy $E$-function is a generalization of the $E$-polynomial, or the Hodge-Deligne polynomial, of an algebraic variety. In \cite{batyrev98} and \cite{bat99}, Batyrev introduced this notion for varieties with log-terminal singularities and for klt pairs. In the case of a projective variety $X$ with at most Gorenstein canonical singularities, the stringy $E$-function is a rational function with integer coefficients (see Section \ref{preliminaries}). If we write it as $E_{st}(X) = E_{st}(X; u,v) = \sum{b_{p,q}u^pv^q}$, the stringy Hodge numbers are defined as: $$h^{p,q}_{st}(X) := (-1)^{p+q}b_{p,q}.$$
Batyrev made the following basic conjecture in \cite[Conjecture 3.10]{batyrev98}.

\begin{intro-conjecture}[Batyrev]\label{conjecture} Let $X$ be a projective variety with Gorenstein canonical singularities. Assume that $E_{st}(X;u,v)$ is a polynomial. Then all stringy Hodge numbers $h_{st}^{p,q}(X)$ are nonnegative.
\end{intro-conjecture}

In \cite{batyrev98} the numbers $h^{p,q}_{st}(X)$ are only called stringy Hodge numbers if $E_{st}(X)$ is a polynomial, and in that case they are similar to the Hodge numbers of a smooth projective variety (see Section \ref{preliminaries}). Batyrev's motivation comes from mirror symmetry and in many examples of interest in this area, the stringy $E$-function is a polynomial. However, this function is defined on a larger class of varieties and, even if it is not a polynomial, the nonnegativity of $h^{p,q}_{st}(X)$ represents a basic numerical constraint on the exceptional divisors in a log-resolution of singularities. Hence, the question is naturally of interest to birational geometry as well.\\

There are several cases in which \conjectureref{conjecture} is known to be true. The first is that of surfaces with canonical singularities. This follows from two facts: stringy Hodge numbers do not change under crepant morphisms as proved by Batyrev (see \theoremref{bat1} below), and every surface with canonical singularities admits a crepant resolution. For toric varieties the Conjecture is also true. It was shown that the stringy $E$-function is a polynomial in \cite{batyrev98}, while \conjectureref{conjecture} was proved in \cite{mp05}. Another case is that of varieties with Gorenstein quotient singularities, where \conjectureref{conjecture} is true as shown in \cite{bat99} for global quotients, and in general in \cite{ya04}. In the general case it was shown that stringy Hodge numbers are related to Hodge numbers in orbifold cohomology. For $\dim{X}=3$ and varieties with terminal isolated singularities of dimension 4 and 5, the answer is also positive. These results hold without the condition of the stringy $E$-function being a polynomial, as was proved in \cite{sv07}. Finally, more examples of classes of isolated singularities where \conjectureref{conjecture} is true can be found in \cite{sc12}.\\

In general, for a resolution of singularities $f:Y \to X$, $$h^{p,0}_{st}(X) = h^{p,0}(Y),$$ and this number is nonnegative. For the remaining stringy Hodge numbers we do not have such an interpretation, and a deeper discussion is needed. We start by looking at $h^{p,1}_{st}(X)$. Given a log-resolution $f:Y\to X$ with exceptional divisor $D$, these numbers have an easy description: $$h^{p,1}_{st}(X) = h^{p,1}(Y) - \sum{h^{p-1,0}(D_i)}$$ where the sum is over all the irreducible components $D_i$ of $D$. The first result does not require $E_{st}(X)$ to be a polynomial.
 
\begin{intro-theorem}\label{hp1} Let $X$ be a projective variety with Gorenstein canonical singularities. Then: $$h^{p,1}_{st}(X) \geq 0$$ for all $p$. 

\end{intro-theorem}

It is a quick application of the fact, proved in \cite{gkkp11}, that given a resolution of singularities with exceptional set $D$ of a variety with klt singularities, a $p$-form defined outside of $D$ extends across $D$ without acquiring any poles.\\

If $E_{st}(X)$ is a polynomial, then it must have degree $2n$, where $n=\dim(X)$; moreover $h_{st}^{p,q}(X) = h^{n-p,n-q}_{st}(X)$ and $h^{p,q}(X) =0$ if $p> n$ or $q> n$ (see Remark \ref{rmkhn}). Therefore, the only non-zero stringy Hodge numbers lie in a Hodge diamond. This symmetry reduces the conjecture to the upper half of the diamond. Thus, if $X$ is a threefold with polynomial stringy $E$-function, it suffices to show that $h^{1,1}_{st}(X)$ and $h^{2,1}_{st}(X)$ are nonnegative. As mentioned above, this was shown without assuming that $E_{st}(X)$ is a polynomial in \cite{sv07}. One consequence of \theoremref{hp1} is a new proof of this fact.\\

More can be said about the diamond for threefolds. As stated above, if $X$ has dimension 3 and has polynomial stringy $E$-function, then $$h^{2,2}_{st}(X) = h^{1,1}_{st}(X).$$ If it is not a polynomial, it can still be shown (\propositionref{h22st}) that: $$h^{2,2}_{st}(X) \geq h^{1,1}_{st}(X).$$ This follows from a new interpretation of $h^{2,2}_{st}(X)- h^{1,1}_{st}(X)$ we give, when $X$ is a threefold, in terms of the analytic local defect of a singularity. This notion was introduced by Kawamata \cite{kawamata88} and it plays an important role in the proof of the existence of a $\QQ$-factorialization of a threefold with terminal singularities. The divisors with discrepancy 1 over the singular points allowed him to conclude that this process ends. A careful analysis of the proof (see \propositionref{bm}) yields the result stated above.\\

 This result is going to be useful in the proof of \theoremref{main}, and this is one of the reasons we are interested in this level of generality. We obtain the following corollary for threefolds.

\begin{intro-corollary}\label{cor3folds} Let $X$ be a threefold with Gorenstein canonical singularities. Then $$h^{p,q}_{st}(X)\geq 0$$ if $p+q\leq 4$. 
\end{intro-corollary}

To study other stringy Hodge numbers when the singular locus of $X$ has high codimension, we follow a strategy similar to that of \cite{sv07}. There the authors used the fact that given a log-resolution of singularities of a variety with isolated singularities, the restriction map in higher cohomologies from the smooth variety to the exceptional divisor is surjective. This result admits a generalization to varieties with singular locus of higher dimension, which in turn leads to the following:

\begin{intro-theorem}\label{hp2s} Let $X$ be a projective variety with at most Gorenstein canonical singularities. Suppose that the singular locus has codimension $c$. Then $$h^{p,2}_{st}(X)\geq 0$$ for $p+2\leq c$.
\end{intro-theorem}

For instance, this shows that applying \theoremref{hp2s} to varieties with singular locus of dimension 1, we obtain that as soon as $\dim{X}\geq 5$ we have $$h^{2,2}_{st}(X)\geq 0.$$ However, it does not apply in general to fourfolds. Using \theoremref{hp1}, the last step in proving \conjectureref{conjecture} when $\dim(X)=4$ is showing the above inequality holds. This is what we focus on next.\\

We can assume $X$ has terminal singularities (see Remark \ref{rmkterminal}) in which case the singular locus has at most dimension 1. Even the case of the product of a threefold with terminal singularities and a smooth curve is not entirely obvious\footnote{It is however if the stringy $E$-function of the threefold is a polynomial.}: one needs \corollaryref{cor3folds} in order to check this inequality (see Section \ref{easyexamples}). It can also be seen that the techniques used for proving \theoremref{hp2s} do not work in general in dimension 4 (see Section \ref{fourfolds}) and therefore we need a new approach. To this end, the strategy is to compare $h^{2,2}_{st}(X)$ with $h^{2,2}_{st}(H)$, for a general hyperplane $H\subseteq X$. We obtain the result under certain conditions, as an application of \corollaryref{cor3folds}. \\

We define the following condition for a log-resolution of singularities $f: Y \to X$ which is an isomorphism outside of the singular locus of $X$:

\begin{align*} (*) &&
\parbox[c]{13cm}{  If $D\subseteq Y$ is the exceptional set of $f$, with irreducible components $D= \bigcup D_i$, $f\restr{D_i}$ has connected fibers, and for any irreducible component $B_{ij}\subseteq D_i\cap D_j$, $f\restr{B_{ij}}$ has connected fibers.}
\end{align*}

\begin{intro-theorem}\label{main}Let $X$ be a projective variety of dimension $4$ with Gorenstein terminal singularities. Suppose there exists a log-resolution of singularities $f:Y\to X$ which satisfies $(*)$.  Then: $$h^{2,2}_{st}(X)\geq 0,$$ and hence $h^{p,q}_{st}(X) \geq 0 $ for all $p+q\leq 4$. Moreover, if $E_{st}(X)$ is a polynomial, then \conjectureref{conjecture} holds for $X$. 

\end{intro-theorem}

The Decomposition Theorem (see \theoremref{dt} below) plays a central role in the proof of the theorem. We use the approach of de Cataldo and Migliorini \cite{cm05}, which, besides giving a decomposition of the cohomologies of $Y$ and $D_i$, provides many tools for understanding its interaction with the Hodge decomposition. This in turn allow us to get a simplified description of $h^{2,2}_{st}(X)$ which can be compared to $h^{2,2}_{st}(H)$, for example, if condition $(*)$ is satisfied. \\ 

Finally, in Section \ref{examples} we discuss a class of fourfolds with Gorenstein terminal singularities to which we can apply \theoremref{main}. Roughly speaking, these are fourfolds that satisfy a strong equisingularity condition with respect to generic hyperplane sections along their singular locus (see \definitionref{equisingular}). In addition, the terminal threefold singularities appearing as such hyperplane sections are required to have a special type of log-resolution (see \definitionref{controlledresolution}), and this class includes: \begin{itemize} \item $A_n$. \item $D_{2n+1}$. \item $E_6$. \end{itemize}
A typical example is the fourfold given by $$x_0x_1x_2 + x_5x_3^2 + x_5x_4^2 + x_3^3 + x_4^3 = 0 $$ in $\PP^5$, which is equisingular along three copies of $\PP^1$ given by $(x_i=x_j=x_3=x_4=0)$ for $\{i,j\}\subseteq\{0,1,2\}$, and the singularities of the hyperplane sections are of type $A_1$ if $x_5\neq 0$.  Other examples include extremal contractions of a smooth fourfold of type $(3,1)$ (that is, the exceptional set is a divisor and its image is a curve) with Gorenstein singularities.\\

If we make some assumptions on the topology of $X$, using the techniques of the proof of \theoremref{hp2s} we obtain the following corollary . It is a consequence of \propositionref{proph5}.
\begin{intro-corollary}Let $X$ be a fourfold with at most Gorestein terminal singularities. If $H^5(X)=0$, or equivalently $H^3(X_{reg}) =0$, then $$h^{2,2}(X)\geq 0.$$ If moreover $E_{st}(X)$ is a polynomial, \conjectureref{conjecture} is true for $X$.
\end{intro-corollary}

\medskip

\noindent
{\bf Acknowledgements.}
I am very grateful to my advisor Mihnea Popa for suggesting the problem, and for his constant support during the project. I am also grateful to Mircea Musta\c{t}\u{a}, János Kollár, Mark de Cataldo, Sándor Kovács, Yajnaseni Dutta, Akash Sengupta and Lei Wu for interesting discussions.

\section{Preliminaries}
\subsection{Stringy Hodge numbers}\label{preliminaries} In this section we give the definition of stringy Hodge numbers, and review some basic results about them. We use the definition given in \cite{batyrev98} in the case the variety is projective.\\

\subsubsection{} We say that a variety $X$ has Gorenstein canonical singularities if it is normal, and the following conditions are satisfied: $X$ has singularities of index 1, which means that $K_X$ is Cartier; and given a log-resolution of singularities $f: Y\to X$, where $D_1, \ldots , D_r$ are the irreducible components of the exceptional set, if we write $$K_Y - f^*(K_X) = \sum{c_iD_i},$$then $$\discrep(D_i, X) := c_i \geq 0$$ for all $i$. These numbers are integers and are usually referred as the discrepancy of $D_i$ with respect to $X$. The two conditions imply that $X$ is Gorenstein, and the second condition defines a variety with canonical singularities. \\
 
The following notation is useful: let $I=\{1, \ldots, r\}$, and for any subset $J\subseteq I$ we define $$\displaystyle D_J = \bigcap_{j\in J}{ D_j}$$ and in case $J=\emptyset$ we put $$D_{\emptyset} = Y.$$ 

\begin{definition}
Let $X$ be a projective variety with Gorenstein canonical singularities. The stringy $E$-function is defined as: $$\displaystyle E_{st}(X;u,v) = \sum_{J\subseteq I}{E(D_J;u,v)\prod_{j\in J}{\frac{uv - (uv)^{a_j +1}}{(uv)^{a_j+1}-1}}}$$ where $a_j = \discrep(D_j, X)$, and for a smooth projective variety $Z$, $$E(Z;u,v) = \sum{(-1)^{p+q}h^{p,q}(Z)u^pv^q}$$ is the Hodge-Deligne polynomial (or $E$-polynomial). We often write $E_{st}(X)$ instead of $E_{st}(X; u,v)$.  
\end{definition}

The key result from Batyrev is that the stringy $E$-function does not depend on the resolution as long as the exceptional set is a divisor and its components with nonzero discrepancy are normal crossings \cite[Theorem 3.4]{batyrev98}.\\

\subsubsection{}  For a projective variety $X$ with Gorenstein canonical singularities we can write: $$E_{st}(X) = \sum{b_{p,q}u^pv^q}$$ with $b_{p,q}\in\ZZ$, as it is a rational function.

\begin{definition}[Stringy Hodge numbers]\label{defstringy} The $(p,q)$-stringy Hodge number is defined as: $$h_{st}^{p,q}(X) = (-1)^{p+q}b_{p,q}.$$
\end{definition}

\begin{remark} In the original definition of stringy Hodge numbers given by Batyrev, it was required the stringy $E$-function to be a polynomial, but as stated above, the nonnegativity of the extended version is interesting from the point of view of birational geometry. 
\end{remark}

\begin{remark}\label{rmkhn} The following are some basic properties of the stringy $E$-function and stringy Hodge numbers that were proved by Batyrev \cite{batyrev98}:
\begin{nenumerate}\item If $X$ is smooth, then $E_{st}(X) = E(X)$ and $$h^{p,q}_{st}(X) = h^{p,q}(X).$$ 
\item The stringy $E$-function is symmetric with respect to its variables: $$E_{st}(X; u,v) = E_{st}(X; v,u),$$ and so $$h^{p,q}_{st}(X) = h^{q,p}_{st}(X).$$ 
\item\label{pd} A Poincaré duality kind of result \cite[Theorem 3.7]{batyrev98}: if $\dim{X} = n$, then $$E_{st}(X; u,v) = (uv)^nE_{st}(X; u^{-1}, v^{-1}).$$ In case $E_{st}(X)$ is a polynomial, this means it must have degree $2n$ and $$h_{st}^{p,q}(X) = h^{n-p,n-q}_{st}(X).$$ 

\end{nenumerate}

\end{remark}

The following is usually applied to make some simplifications:

\begin{theorem}[{\cite[Theorem 3.12]{batyrev98}}]\label{bat1} Assume that an algebraic variety with Gorenstein canonical singularities $X$ admits a projective birational morphism $f:Y\to X$ such that $f^*K_X = K_Y$. Then $E_{st}(X) = E_{st}(Y)$. In particular the stringy Hodge numbers are the same.

\end{theorem}

\begin{remark}\label{rmkterminal} For a variety $X$ with canonical singularities, there exists a crepant birational morphism $f:Y \to X$ such that $Y$ has at most terminal singularities \cite[Corollary 1.4.3]{bchm10}. This was proved in dimension 3 by Reid \cite[Main Theorem]{reid83}. Using \theoremref{bat1} we can assume that $X$ has at most terminal singularities.
\end{remark}

\subsubsection{}\label{d(p)} The following is useful for simplifying various expressions.
\begin{notation} Let $f: Y \to X$ be a log-resolution of singularities of $X$ and $\bigcup_{i\in I} D_i = D\subseteq Y$ be the reduced exceptional divisor. We denote: $$D(p) := \coprod_{\substack{J\subseteq I \\ |J|=p}}D_J.$$
\end{notation}

\begin{remark}\label{descriptionshn}  Let $X$ be a projective variety with at most Gorenstein terminal singularities. The following are explicit descriptions of some stringy Hodge numbers: 
%Under these conditions and using that $\frac{uv - (uv)^{a_i +1}}{(uv)^{a_i+1}-1}= -uv + (uv)^{a_i +1} - (uv)^{a_i +2} + (uv)^{2(a_i +1)} + \ldots$ we can calculate some of them:
 \begin{itemize}
	\item $h_{st}^{p,0}(X) = h^{p,0}(Y).$
	\item $h_{st}^{p,1}(X) = h^{p,1}(Y) -  h^{p-1,0}(D(1))$.
	\item $h_{st}^{p,2}(X) = \displaystyle h^{p,2}(Y) - h^{p-1,1}(D(1)) + h^{p-2,0}(D(2)) + \sum_{a_j = 1}{h^{p-2,0}(D_j)} $\\ where $a_j = \discrep(D_j, X)$. 
\end{itemize}
Note that as $D(p)$ is a disjoint union of smooth projective varieties, its cohomology spaces are direct sums of the cohomologies of its components. For example $H^k(D(1)) = \bigoplus H^k(D_j)$. 
\end{remark}

\begin{notation}\label{defa} There is always a piece of the summands of stringy Hodge numbers in which discrepancies do not show up. That is, $$h^{p,q}_{st}(X) = \sum{(-1)^kh^{p-k,q-k}(D(k)) } + \text{ extra terms}, $$ so it is convenient to define \begin{equation}\label{apq}a_{p,q}(X) :=\sum{(-1)^kh^{p-k,q-k}(D(k))}. \end{equation}
\end{notation}

\subsection{Background results in mixed Hodge theory and perverse sheaves}
This section contains most of the known facts that are used later in the paper, for easy reference.

\subsubsection{}\label{logcx} We start with some results about the logarithmic complex; see e.g. \cite[Section 8]{voisin1}. They are used in the proof of \theoremref{hp1} keeping the same notation.\\

Let $Y$ be a smooth projective variety and $D$ a simple normal crossings divisor in $Y$. 

\begin{definition}\label{weightfiltration} Let $\Omega_Y^p(\log{D})$ be the sheaf of  $p$-forms with logarithmic poles along $D$. The increasing weight filtration consists of subsheaves $$W_k\Omega_Y^p(\log{D}) \subseteq \Omega_Y^p(\log{D})$$ such that if $z_1, \ldots, z_n$ are local coordinates on an open set $V$, and $D$ is given by the equation $$z_1\cdots z_r=0,$$ then $W_k\Omega^p(\log{D})\restr{W}$ is a free $\shf{O}_W$ module generated by elements of the form $$\frac{dz_{i_1}}{z_{i_1}}\wedge \cdots \wedge \frac{dz_{i_s}}{z_{i_s}}\wedge dz_{j_1} \wedge \cdots \wedge dz_{j_{p-s}}$$ with $i_l\leq r$ and $s\leq k$.
\end{definition}

The sheaves of logarithmic poles form the logarithmic sequence $\Omega^{\bullet}_Y(\log{D})$, and one crucial result is the isomorphism 
\begin{equation}\label{logarthmichyp} H^k(Y\setminus D, \CC) \cong \HH^k(Y, \Omega^{\bullet}_Y(\log{D}))
\end{equation}
due to Deligne \cite{HodgeII}*{Proposition 3.1.8}.\\

The weight filtration on the sheaves induces one on the sequence, which in turn defines a weight filtration on the cohomology of $Y\setminus D$, and it is the same as the one discussed in Section \ref{mhs}. \\

Using the same notation as in Section \ref{d(p)} and letting $j_k: D(k) \into Y$ we have the following (see e.g. \cite[Proposition 8.32]{voisin1}):

\begin{proposition}\label{logarithmicquotient} There exists a natural isomorphism $$W_k\Omega^{p}_Y(\log{D}) / W_{k-1}\Omega^{p}_Y(\log{D}) \cong j_{k*}\Omega^{p-k}_{D(k)}.$$
\end{proposition}

\subsubsection{}\label{mhs} The cohomology spaces of an algebraic variety are endowed with a mixed Hodge structure functorial with respect to algebraic morphisms, as proved by Deligne. We state a few facts used in several parts of this paper (see e.g. \cite[Section 3]{elzeinetal}).\\

Let $X$ be an algebraic variety of dimension $n$. For simplicity all cohomologies are assumed to be with complex coefficients. There exists an increasing filtration $$\{0\}=W_{-1} \subseteq W_0 \subseteq \cdots \subseteq W_{2k}= H^k(X)$$ of vector subspaces called the weight filtration. The graded pieces are defined to be $${\rm Gr}^W_jH^k(X) = W_j/W_{j-1}.$$ There also exists a decreasing filtration $$ \{0\} = F^m \subseteq F^{m-1} \subseteq \cdots \subseteq F^0 = H^k(X)$$ of vector subspaces, called the Hodge filtration, which induces a pure Hodge structure of weight $j$ on ${\rm Gr}^W_jH^k(X)$. 

\begin{remark}\label{weights} The following are basic facts about the weights of the cohomology $H^k(X)$: 
\begin{renumerate} \item\label{mhssmooth} Suppose $X$ is a smooth variety. Then $W_{k-1} = {0}$, that is, it only has the ``upper weights".
\item\label{mhsproper} Suppose $X$ is a proper variety. Then $W_k = H^k(X)$, that is, it only has the ``lower weights".

\end{renumerate}
\end{remark}

\subsubsection{}\label{mhssnc} The mixed Hodge structure of a simple normal crossings variety can be described in a simple way. This computation is used, for example, in the proof of \theoremref{hp2s}. See e.g. \cite[Part II, 1]{elz83} and \cite[Section 3]{elzeinetal}.\\

Let $D$ be a simple normal crossings variety and $H^k\big(D(r)\big) \stackrel{\delta_r}{\longrightarrow} H^k\big(D(r+1)\big)$ be the alternating sum of the pullbacks of the natural inclusions $D(r+1) \into D(r)$. We get the complex
$$ 0 \longrightarrow H^k\big(D(1)\big) \stackrel{\delta_1}{\longrightarrow} H^k\big(D(2)\big) \stackrel{\delta_2}{\longrightarrow} \cdots \stackrel{\delta_l}{\longrightarrow} 
H^k\big(D(l+1)\big) \stackrel{\delta_{l+1}}{\longrightarrow} \cdots,$$
in which all cohomologies have $\CC$-coefficients. The weight $k$ piece of the mixed Hodge structure on the cohomology of $D$ are the cohomologies of this complex. More precisely, we have  \begin{equation}\label{snchp}{\rm Gr}^W_k H^{k+l}(D) = \ker\delta_{l+1}/\im{\delta_l}.\end{equation}
The Hodge space $H^{p,q}\big({\rm Gr}^W_k H^{k+l}(D)\big)$ is obtained by applying first $H^{p,q}$ to the complex and then taking the cohomologies.

\subsubsection{} The Decomposition Theorem and related results are used in the proof of \theoremref{main}. We gather them following the notation of de Cataldo and Migliorini \cite{cm05}.\\

For an algebraic variety $X$ we denote ${\rm D}(X)$ the category of constructible complexes, and by $\Perv(X)$ the abelian subcategory of perverse sheaves, together with the perverse cohomology functors $$^{\pp}\cohH^k : {\rm D}(X) \to \Perv(X).$$
The simple objects of $\Perv(X)$ are the intersection cohomology complexes $IC_X(L)$, where $L$ is a local system in $U\subseteq X_{reg}$ . An explicit definition can be found in \cite[Proposition 8.2.11]{hottaetal}. When the local system $L= \CC_U$, we simply denote the intersection cohomology complex as $IC_X$.\\

\begin{theorem}[Decomposition theorem]\label{dt} Let $f: Y \to X$ be a projective morphism from a smooth variety $Y$ of dimension $n$. Then: $$\derR f_*\CC_Y[n] \cong \bigoplus{^{\mathfrak{p}}\mathcal{H}^k(\derR f_*\CC_Y[n])[-k]} $$ in ${\rm D}(X)$. Moreover, each $^{\mathfrak{p}}\mathcal{H}^k(\derR f_*\CC_Y[n])$ can be decomposed into intersection complexes of the strata. 
\end{theorem}

A direct consequence of the theorem is that we get a direct sum decomposition of the cohomology spaces of $Y$, by taking hypercohomology functors. More will be said about this decomposition in Section \ref{perversecoh}. \\

The next results are used when comparing the direct sum decomposition of the cohomologies of $Y$ to the one we get on certain subvarieties, by applying the theorem to the restriction map. They work in general for a normally nonsingular inclusion, but its definition is rather technical and can be found in \cite[Section 3.5]{cm05}. Good examples are subvarieties of the ambient projective space of a projective variety, intersecting all of the strata (of a given stratification) transversely. The example we are interested in is a general hyperplane $H\subseteq X$.

\begin{lemma}[{\cite[Lemma 4.3.8]{cm05}}]\label{lemmahyp} Let 
\begin{equation}\begin{tikzcd}
{H'}^m \arrow[r, "v"] \arrow[d, "f'"]
& Y^n \arrow[d, "f"] \\
H \arrow[r, "u"]
& X
\end{tikzcd}\end{equation}
be a Cartesian diagram of maps of algebraic varieties of the indicated dimensions and $f$ proper. Assume that $u$ is a normally nonsingular inclusion. Then $$^{\pp}\cohH^i(\derR f'_*\CC_{H'}[m])\cong u^{*\pp}\cohH^i(\derR f_*\CC_Y[n])[m-n]$$ for every $i\in\ZZ$ and the natural map $$v^*: H^k(Y) \to H^k(H')$$ is compatible with the direct sum decomposition in perverse cohomology groups and it is strict.

\end{lemma}

%As $H$ doesn't intersect $S$ we have that $u^*j_*L_{U,-1} = u^*L_{U,-1}$ which is the restriction of this local system to $H$. Hence Lemma \ref{lemmahyp} says that $$u^*L_{U,-1} = L_{U,-1}\restr{T} \cong \bigoplus_{i=1}^k{H_4(f^{-1}(x_i))}.$$

\begin{proposition}[{\cite[Proposition 4.7.7]{cm05}}]\label{prophyp} Let $P\in \Perv(X)$. Then the natural map $$u^*: \HH^j(X,P) \to \HH^j(H, u^*P)$$ is an isomorphism for $j\leq -2$ and injective for $j=1$.
\end{proposition} 

\subsubsection{}\label{perversecoh} We now consider perverse cohomology groups.

 \begin{definition}\label{perverse cohomology} Using the same notation as in \theoremref{dt} we define the subspaces: $$H^{n+l}_b(Y):= \HH^{l-b}(X,\ ^{\pp}\cohH^b(\derR f_*\CC_Y[n])).$$
\end{definition}

We have a direct sum decomposition $$H^k(Y) =\bigoplus H^k_b(Y)$$ which can be made into a decomposition by Hodge substructures \cite[Remark 2.1.6]{cm05}. Understanding these spaces is the main task in the proof of \theoremref{main} and one of the tools we use is:

\begin{theorem}[Global invariant cycle theorem]\label{gict} Suppose $f: Z \to U$ is a smooth projective map and let $\bar{Z}$ be a smooth compactification of $Z$. Then for $x\in U$, $$H^0(U, R^kf_*\CC_Z) = H^k(f^{-1}(x))^{\pi_1(U,x)} = \im\{H^k(\bar{Z}) \to H^k(f^{-1}(x))\}.$$

\end{theorem}

Finally, we have the following analogues of classical results in Hodge theory \cite[Theorem 2.1.4, Theorem 2.2.3]{cm05}:

\begin{theorem}[Hard Lefschetz Theorem for Perverse cohomology groups]\label{hlpc} Let $k\geq 0$ and $b,j \in \ZZ$. Then the following cup product maps are isomorphisms: 
\begin{renumerate}\item\label{hla} $\eta^k: H^j_{-k}(Y) \cong H^{j+2k}_{k}(Y)$ where $\eta\in H^2(Y)$ is an ample class.
\item\label{hlpa} $L^k: H^{n+b-k}_b(Y) \cong H^{n+b+k}_{b}(Y)$ where $L=f^*A\in H^2(Y)$ for an ample class $A\in H^2(X)$.
\end{renumerate}

\end{theorem}

\begin{theorem}[Weak Lefschetz Theorem for Intersection Cohomology]\label{wlt} Let $u: H\into X$ be a general hyperplane section of a projective variety. Then $$u^*: IH^j(X) \to IH^j(H)$$ is an isomorphism for $j\leq \dim{X} -2$ and injective for $j= \dim{X} -1$.
\end{theorem}

\begin{theorem}[Hard Lefschetz Theorem for Intersection Cohomology]\label{hlt} Let $X$ be a projective variety. There is an isomorphism $$A^j:  IH^{\dim{X} -j}(X) \to IH^{\dim{X} + j}(X)$$ where $A$ is an ample class in $X$.
\end{theorem}

\section{First results}
\subsection{$(p,1)$-stringy Hodge numbers} In this section we prove \theoremref{hp1}. The result is an application of \theoremref{logarithmicquotient} and \propositionref{gkkp11}. Recall that for a projective variety with at most Gorenstein canonical singularities, the $(p,1)$-stringy Hodge number is defined as: $$h^{p,1}_{st}(X) = h^{p,1}(Y) - h^{p-1,0}(D(1))$$ where $f:Y\to X$ is a log-resolution of singularities and $D$ the exceptional set.\\

%\begin{equation}\label{les} 0 \to \Omega^p_Y \to \Omega^p_Y(\log{D}) \to \bigoplus\Omega^{p-1}_{D_i}(\log{D_i^c}) \to \bigoplus\Omega^{p-2}_{D_{ij}}(\log{D_{ij}^c}) \to \ldots \end{equation}

%where $D_{ij}^c = (D-D_i-D_j)\big|_{D_{ij}}$, the first map is the inclusion and the following maps are induced by the residue maps. By $\Omega^{p-1}_{D_i}(\log{D_i^c})$ we actually mean $j_*\Omega^{p-1}_{D_i}(\log{D_i^c})$ where $j:D_i \hookrightarrow Y$. These sheaves have a weight filtration and the image of residue map goes to one weight less. Giving a trivial weight filtration to $\Omega^p_X$, the exact sequence induces a long exact sequence on the weight pieces, that is, for each $k$ we get the following exact sequence:
%\begin{equation} 0 \to {\Omega^p_Y \to W_k\Omega^p_Y(\log{D}) \to \bigoplus W_{k-1}\Omega^{p-1}_{D_i}(\log{D_i^c}) \to} \to {\bigoplus W_{k-2}\Omega^{p-2}_{D_{ij}}(\log{D_{ij}^c}) \to \ldots}
%\end{equation}
%The weight one piece gives us the usual 

We use the notation in Section \ref{logcx}. By \propositionref{logarithmicquotient}, we have the following exact sequence: \begin{equation}0 \to \Omega^p_Y \to W_1\Omega^p_Y(\log{D}) \to \bigoplus\Omega^{p-1}_{D_i} \to 0. \end{equation}
Consider the long exact sequence of cohomologies: \begin{equation}\label{eq1}0 \to H^0(\Omega^p_X) \to H^0(W_1\Omega^p_X(\log{D})) \to \bigoplus H^0(\Omega^{p-1}_{D_i}) \stackrel{d'}{\to} H^1(\Omega^p_X) \to \ldots\end{equation}

%\begin{remark} If $d'$ is injective, then $h_{st}^{p,1}$ is nonnegative. 
%\end{remark}
We use the following simplified version of \cite[Proposition 19.1]{gkkp11}.  

\begin{proposition} \label{gkkp11} Let $X$ be a complex quasi-projective variety with at most canonical singularities, and $f: Y \to X$ a log-resolution with exceptional set $D\subseteq Y$. Then the natural injection $$H^0(Y, \Omega^p_Y) \to H^0(Y, \Omega^p_Y(\log{D}))$$ is an isomorphism.
\end{proposition}

\begin{proof}[Proof of \theoremref{hp1}] The inclusion $H^0(\Omega^p_Y) \into H^0(W_1\Omega^p_Y(\log{D}))$ is actually an isomorphism by \propositionref{gkkp11}. Indeed, $H^0(\Omega^p_Y) \into H^0(W_1\Omega^p_Y(\log{D}))$ factors through $H^0(\Omega^p_Y) \into H^0(\Omega^p_Y(\log{D}))$ which is an isomorphism. Hence, $d'$ in (\ref{eq1}) is injective. Taking dimensions we get: $$h^{p,1}(Y) - h^{p-1,0}(D(1)) \geq 0,$$ which is precisely the statement $$h^{p,1}_{st}(X)\geq 0.$$
\end{proof}

This result, together with the description of $h^{p,0}_{st}(X)$ (see Section \ref{descriptionshn}), implies the following corollary. It was originally proved in \cite{sv07} with a different method.

\begin{corollary} If $X$ is a projective threefold with Gorenstein canonical singularities, then $$h^{p,q}_{st}(X)\geq 0$$ for $p+q\leq 3$. In particular, \conjectureref{conjecture} holds for $\dim{X} = 3$. 

\end{corollary}

\subsection{Threefolds}\label{3folds} In this section we prove \propositionref{h22st} which is used in the proof of \theoremref{main}, and expands the information about the stringy Hodge diamond when $\dim{X} = 3$.\\

Let $X$ be a projective threefold with at most Gorenstein terminal singularities. It is enough to consider this case by Remark \ref{rmkterminal}. If $E_{st}(X)$ is a polynomial, one of the consequences is that $$h^{2,2}_{st}(X) = h^{1,1}_{st}(X)$$ by Remark \ref{rmkhn}\ref{pd}. As it was shown that $h^{1,1}_{st}(X)\geq 0$, the same is true for $h^{2,2}_{st}(X)$. If $E_{st}(X)$ is not a polynomial, then few things are known about the lower stringy Hodge triangle. For example, some numbers might fail to be nonnegative; see the example of a threefold with $h^{3,3}_{st}(X) < 0$ in \cite[Remark 3.5(1)]{sv07}.\\

We give a new interpretation of the number $h^{2,2}_{st}(X) - h^{1,1}_{st}(X)$ in terms of the analytic local defect, and use the existence of an analytic $\QQ$-factorialization of a threefold terminal singularity in \cite{kawamata88}, to show that $$h^{2,2}_{st}(X)\geq 0.$$ 

\begin{definition} Let $H$ be a threefold with isolated singularities. For a point $x\in H$ we define the local defect of $H$ at $x$ as $$\sigma(H,x) = \dim{\Weil(U_x)_{\QQ}/\Cart(U_x)_{\QQ}}$$ where $U_x$ is a contractible Stein open neighborhood of $x$ in $H$. If $\sigma(H,x) = 0$, we say that $H$ is analytically $\QQ$-factorial at $x$.
\end{definition}

If $H$ has rational singularities, the local defect is finite at every point \cite[Lemma 1.12]{kawamata88}. Using the strategy of the proof of this fact we can get a convenient description of $\sigma(H,x)$. Let $\{x_1, \ldots, x_m\}$ be the isolated singularities of $H$, and $x_i\in U_i$ be a Stein contractible neighborhood. For a log-resolution $f: H'\to H$, we have isomorphisms $$\Divi^0(U_i) \cong \Divi^0(V_i)$$ where $V_i = f^{-1}(U_i)$. As $U_i$ is Stein, we can apply Theorem B of Cartan \cite{cartan}, and obtain that $\Pic(U_i) = 0$. Let $f^{-1}(x_i) = \bigcup E_{i,j}$. Using the sequence $$\ZZ[E_{i,j}] \to \Weil(V_i) \to \Weil(V_i\setminus E_i) \to 0,$$ and the fact that $$\Weil(V_i\setminus E_i) \cong \Weil(U_i),$$ we get the following isomorphism given in \cite[3.9]{ns95}: $$\Weil(U_i)/\Cart(U_i) \cong H^1(V_i, O^*_{V_i})/\sum\ZZ[E_{i,j}].$$
Using again the fact that $H$ has rational singularities, we conclude that $$H^1(V_i, O^*_{V_i}) \cong H^2(E_i, \ZZ).$$ Therefore, we have the following expression in terms of the cohomology of $E_i$: \begin{equation}\label{localdefect}\sigma(H, x_i) = \dim{H^2(E_i)/\sum\CC[E_{i,j}]}.\end{equation}

\begin{proposition}\label{bm} Let $H$ be a threefold with Gorenstein terminal singularities and $x\in H$ a singular point. Then $$\sigma(H,x) \leq \#\{E_j : \ \discrep(E_j, H) = 1 \ \text{ and } \Center_H(E_j)=x  \}.$$ 
\end{proposition}

\begin{proof}We start with the same set up as in the proof of \cite[Corollary 4.5]{kawamata88}. Let $V=V_0$ be a Stein contractible neighborhood of $x$. Suppose it is not $\QQ$-factorial and let $D_0\in \Weil(V_0)$ which is not $\QQ$-Cartier. There exists an analytic space and a projective morphism $f_0: V_1 \to V_0$ which is an isomorphism in codimension 1 and the strict transform of $D_0$ is $\QQ$-Cartier \cite[Theorem 4.1 and Lemma 3.1]{kawamata88}. $V_1$ has also terminal singularities and $\Pic(V_1)\cong H^2(V_1,\ZZ)$. We get $$\sigma(V_0,x) = \sigma(V_1, f_0^{-1}(x)) + b_2(V_1)$$ where  $\sigma(V_1, f_0^{-1}(x)) = \dim{\Weil(V_1)_{\QQ}/\Cart(V_1)_{\QQ}}$. $V_1$ retracts to the exceptional set of $f_0$ which consists of a union of curves. The blow up of the any of these curves (which are smooth in $V_1$) produce a discrepancy 1 divisor over $x$ \cite[Proof of Theorem 6.25]{km}, and each of them have a different center in $V_1$. The conclusion is that the local defect and the number of discrepancy 1 divisors over $x$ drop by the same number. \\

We continue the process by induction. Suppose $V_n$ is not $\QQ$-factorial and let $D_n\in \Weil(V_n)$ which is not $\QQ$-Cartier. As before, there exists a projective morphism $f_n: V_{n+1}\to V_n$ with the same conditions as $f_0$ and $D_0$, in particular the exceptional set is a union of curves. Let $g_n: V_n \to V_0$ be the composition of the previous morphisms. The exceptional set is also a union of curves, and $V_n$ retract onto them. As $f_n$ is an isomorphism outside of the isolated singularities of $V_n$, the class of each of the curves $C_i$ in $H^2(V_n, \ZZ)$ maps to the class of the curve $f_{n*}^{-1}C_i$ in $H^2(V_{n+1}, \ZZ)$. Therefore,  $$\sigma(V_n,g_n^{-1}(x)) = \sigma(V_{n+1}, g_{n+1}^{-1}(x)) + b_2(V_{n+1}) - b_2(V_{n}),$$ that is, it drops by the number of exceptional curves of $f_n$. They also correspond to discrepancy 1 divisors over $x$ with different centers in $V_{n+1}$. So at each step, the local defect is dropping by the number of discrepancy 1 divisors over $x$ that come from blowing up the exceptional curves of the map.\\

The process stops and we reach an analytic space $W$ which is $\QQ$-factorial. For $W$ the claim is trivial as the local defect is 0. The result follows.

\end{proof}

Every irreducible component $E_{i,j}$ of an exceptional divisor over a terminal singularity of a threefold is birational to a ruled surface \cite[Corollary 2.14]{reid80}. This implies $b_2(E_{i,j}) = h^{1,1}(E_{i,j})$, as $h^{2,0}$ is a birational invariant. \\

The vector space $H^2(E_i)$ has a pure Hodge structure (cf. proof of Proposition 2.1 in \cite{ns95}), which implies that $$b_2(E_i) = b_2(E_i(1)) - b_2(E_i(2))$$ (see Section \ref{mhssnc}), and therefore $$b_2(E_i) = h^{1,1}(E_i(1)) - h^0(E(2)).$$ 
%The classes $\CC[E_{i,j}] \subseteq H^{1,1}(H')$ can be seen as the inclusion $$H_4(E_i) \into H_4(H') \cong H^2(H').$$ To understand the restriction to $H^2(E_i)$ we can use what is explained in Example 2.4 of \cite{cm05} together with \cite[Theorem 2.1.10]{cm05} to conclude that 
We also have an inclusion $$H_4(E_i) \into H^2(E_i),$$ which means the classes $\CC[E_{i,j}]$ are independent in $H^2(E_i)$ (see for example \cite[Corollary 1.12]{St83}). Plugging in this information into (\ref{localdefect}) we get: \begin{equation}\label{localdefect2}\sigma(H,x_i) = h^{1,1}(E_{i}(1)) - h^0(E_{i}(2)) - h^0(E_i(1)).\end{equation}

\begin{proposition}\label{h22st} For a projective threefold $H$ with at most Gorenstein canonical singularities $$h^{2,2}_{st}(H) \geq h^{1,1}_{st}(H).$$

\end{proposition}

\begin{proof} We can assume $H$ has terminal singularities by Remark \ref{rmkterminal}. Let $f': H'\to H$ be a log-resolution with exceptional set $E= \bigcup E_{i,j}$. Recall the description of the stringy Hodge numbers in Remark \ref{descriptionshn}. We have:  $$h^{2,2}_{st}(H) - h^{1,1}_{st}(H) = -h^{1,1}(E(1)) + h^0(E(2)) + h^0(E(1))+ \sum_{a_{ij}=1}h^0(E_{i,j}).$$ Note that all exceptional divisors of discrepancy 1 must show up in $f'$ \cite[Proof of Lemma 6.36]{km}. By taking addition over every singularity, \propositionref{bm} together with (\ref{localdefect2}) is equivalent to $$h^{2,2}_{st}(H) - h^{1,1}_{st}(H) \geq 0.$$ 

\end{proof}

\begin{proof}[Proof of \corollaryref{cor3folds}] This is a consequence of \theoremref{hp1} and \propositionref{h22st}.
\end{proof}

\subsection{Low dimensional singular locus} In this section we prove \theoremref{hp2s}. We in fact prove a more general result. Recall that $h^{p,q}_{st}(X)$ can be written as the sum of two pieces: one (called $a_{p,q}(X)$ in (\ref{apq})) which does not depend on the discrepancy of the irreducible components of the exceptional divisor, and another one which does. We show that $a_{p,q}(X)$ is nonnegative under some conditions of the dimension of the singular locus.\\

If a variety has isolated singularities, the higher cohomologies of the exceptional divisor have a pure Hodge structure \cite[Corollary 1.12]{St83}. The following is a generalization without restriction on the dimension of the singular locus:

\begin{theorem}[{\cite[Theorem 6.31]{PS}}] \label{lemmaps}Let $X$ be an algebraic variety of dimension $n$. Let $Z$ be the singular locus which has dimension $s$ and $f: Y\to X$ a resolution such that $f^{-1}(Z) = D$ is a simple normal crossings divisor on $Y$. Then:
\begin{itemize}
	\item $W_{k-1}H^k(D) = 0$ for all $k\geq n+s$.
	\item If $Z$ is compact, then $H^k(D)$ is pure of weight $k$ for all $k\geq n+s$.
\end{itemize}
\end{theorem}

We reproduce the proof from \cite{PS}, as we need to use some of its intermediate steps in the proof of \propositionref{prop}.

\begin{proof}Let $U_0, \ldots , U_s$ be $(s+1)$ affine open subsets of $X$ which cover $Z$. In the projective case we can take $U_i$ to be the complement of a general ample divisor. Indeed, as the intersection of these divisors is a variety of codimension $s$, and the divisors are general, this subvariety does not intersect $Z$, and hence $U:= U_0 \cup \ldots \cup U_s$ covers $Z$. As $U_i$ is affine, it is homotopic to a CW-complex of dimension $n$, and then $H^k(U_i) = 0$ for $k>n$; see e.g. \cite[Part II, 5.1*]{stratifiedmt}. Using Mayer-Vietoris and induction we get that $$H^k(U) = 0$$ for $k>n+s$. \\

Let $\tilde{U} = f^{-1}(U)$. For the birational morphism $f\restr{\tilde{U}}$ we have the following long exact sequence:
$$\ldots \to H^k(U) \to H^k(\tilde{U}) \oplus H^k(Z) \to H^k(D) \to H^{k+1}(U) \to \ldots$$ (see for instance \cite[Corollary-Definition 5.37]{PS}).
For $k\geq n+s$ we get a surjection  \begin{equation}\label{surj} H^k(\tilde{U}) \to H^k(D).\end{equation}
As $\tilde{U}$ is smooth, it only has the upper weights (Remark \ref{weights}\ref{mhssmooth}). From this we get the first result. If moreover $Z$ is compact, then $D$ is compact as well, and hence it only has lower weights (Remark \ref{weights}\ref{mhsproper}). The result follows.

\end{proof}

\begin{proposition}\label{prop} Let $X$ be a projective variety of dimension $n$ with Gorenstein canonical singularities. Let $Z$ be its singular locus and suppose it has dimension $s$. Then $a_{p,q}(X)$ is nonnegative for all $p,q$ such that $p+q\leq n-s$.
\end{proposition}

\begin{proof} Let $f: Y\to X$ be a log-resolution and $f^{-1}(Z) = D$ a simple normal crossings divisor. Let $k:= 2n -p - q$. Note that by assumption we have that $k\geq n+s$. By \theoremref{lemmaps} we get that $H^k(D)$ has a pure Hodge structure. The complex used to describe the mixed Hodge structure of $H^k(D)$ (see Section \ref{mhssnc}) is then an exact sequence:
$$ 0\to H^k(D) \to H^k(D(1)) \to \ldots \to H^k(D(p)) \to 0.$$
Taking the $h^{n-p,n-q}$ pieces we get $$h^{p,q}(D) = \sum_{i\geq 1}{(-1)^{i+1}h^{n-p,n-q}(D(i))} = \sum_{i\geq 1}{(-1)^{i+1}h^{p-i,q-i}(D(i))}.$$

Let $U$ be as in the proof of \theoremref{lemmaps}. We have a surjection $H^k(\tilde{U}) \to H^k(D)$. As $Y$ is a smooth compactification of $\tilde{U}$, we have that the image of $H^k(Y) \to H^k(D)$ is the same as that of $H^k(\tilde{U}) \to H^k(D)$ \cite[Prop. 8.2.6]{HodgeIII}. In particular, we obtain that $$h^{p,q}(Y) \geq h^{p,q}(D) = \sum_{i\geq 1}{(-1)^{i+1}h^{p-i,q-i}(D(i))}, $$
and therefore $$a_{p,q}(X) = \sum_{i\geq 0}{(-1)^ih^{p-i,q-i}(D(i))} \geq 0.$$

\end{proof}

\begin{proof}[Proof of \theoremref{hp2s}]This is a direct consequence of \propositionref{prop}. Recall that $$h^{p,2}_{st}(X) = a_{p,2}(X) + \sum_{a_j=1}h^0(D_j).$$ Since $a_{p,2}(X)\geq 0$, we have $$h^{p,2}_{st}(X)\geq 0.$$
\end{proof}

\section{Main theorem}
\subsection{Fourfolds}\label{fourfolds}
A consequence of \theoremref{hp1} applied to a projective fourfold $X$ with Gorenstein terminal singularities, together with the fact that $h^{p,0}_{st}(X)\geq 0$ always holds (see Remark \ref{descriptionshn}), is that $$h^{p,q}_{st}(X)\geq 0$$ when $p+q\leq 4$ and $(p,q)\neq (2,2)$. This means that for \conjectureref{conjecture} to hold, it remains to prove $h^{2,2}_{st}(X)\geq 0$. If $X$ has isolated singularities, \theoremref{hp2s} implies the inequality above. Therefore, for the rest of the paper we discuss fourfolds whose singular locus has at least one component of dimension 1.\\

Let $f:Y \to X$ be a log-resolution of singularities, $D$ the exceptional set, and $C$ the singular locus of $X$. Recall that $$h^{2,2}_{st}(X) = a_{2,2}(X) + \sum_{a_j=1}{h^0(D_j)}$$ and $$a_{2,2}(X) = h^{2,2}(Y) - h^{1,1}(D(1)) + h^0(D(2)).$$ Using the isomorphism (\ref{logarthmichyp}), and denoting $U:= Y\setminus D \cong X_{reg}$, we have $$a_{2,2}(X) = h^{2,2}({\rm Gr}^W_4H^4(U)) - h^{2,2}({\rm Gr}^W_4H^3(U)) + h^{2,2}({\rm Gr}^W_4H^2(U))\footnote{This last term $h^{2,2}({\rm Gr}^W_4H^2(U))$ can actually be shown to be equal to 0, but this does not affect the discussion below.}.$$ Indeed, the complex $E_1^{*,4}$ given by the spectral sequence used to compute the weight filtration on the cohomologies of $U$ is precisely $$0\to H^0(D(2)) \to H^2(D(1)) \to H^4(Y) \to 0,$$ where the maps are alternating sums of the Gysin morphisms (see \cite[Corollary 8.33 and Proposition 8.34]{voisin1}). These are maps of Hodge structures, and the spectral sequence degenerates at $E_2$ with $E_2^{p,q} = {\rm Gr}^W_qH^{p+q}(U)$ (see \cite[Theorem 8.35]{voisin1}). The equality above follows by taking the corresponding Hodge pieces.\\

A sufficient condition for $a_{2,2}(X)$ to be nonnegative, and as a consequence $h_{st}^{2,2}(X)$ as well, is $$h^{2,2}({\rm Gr}^W_4H^3(U)) = 0.$$ This is equivalent by Poincaré duality to $$h^{2,2}({\rm Gr}^W_4H^5_c(U)) = 0.$$ %Using the long exact sequence of the pair $(Y,D)$, which a sequence of mixed Hodge structures \cite[Proposition 5.54]{PS}, and taking $Gr_4$ in the sequence we get $$0\to Gr_4H^4_c(U) \to H^4(Y) \to Gr_4H^4(D) \to Gr_4H^5_c(U) \to 0.$$  The conclusion is that $h^{2,2}(Gr_4H^5_c(U)) = 0$ is equivalent to the map $$H^{2,2}(Y) \to H^{2,2}(Gr_4H^4(D))$$ being surjective. 
On the other hand, using the sequence of the pair $(X, C)$ we obtain $$H^5_c(U) \cong H^5(X)$$ as mixed Hodge structures. This means that an equivalent sufficient condition is $h^{2,2}({\rm Gr}^W_4H^5(X)) = 0$. With this discussion we get the following.

\begin{proposition}\label{proph5} Let $X$ be a fourfold with at most Gorenstein terminal singularities. If $W_4H^5(X)=0$, then $$h^{2,2}_{st}(X)\geq 0.$$ If moreover $E_{st}(X)$ is a polynomial, \conjectureref{conjecture} is true for $X$. In particular, if $H^5(X) = 0$, or equivalently $H^3(X_{reg}) = 0$, all of the above holds. 
\end{proposition}

The following example shows that $a_{2,2}(X)$ can be negative, and therefore arguments to prove that in general $h^{2,2}_{st}(X)\geq 0$ must take into consideration the term $\displaystyle\sum_{a_j=1}h^0(D_j)$.

\begin{example} Let $X_0$ be the Burkhardt quartic, given by the equation $$x_0^4-x_0(x_1^3+x_2^3+x_3^3+x_4^3)+3x_1x_2x_3x_4 = 0$$ in $\PP^4$. This threefold has 45 nodes, and blowing them up we get $Y_0$ which is a smooth rational variety \cite[Corollary 2.5]{hw01} with $b_2(Y_0) = 61$ \cite[Corollary 2.12]{hw01}. As $Y_0$ is rational, $$h^{2,0}(Y_0)=0,$$ hence $$h^{2,2}(Y_0)= h^{1,1}(Y_0) = 61.$$
Moreover, the exceptional divisor over every node is isomorphic to $\PP^1 \times \PP^1$. Let $X = X_0\times \PP^1$. As $X_0$ only has nodes as singularities, $E_{st}(X_0)$ is a polynomial \cite{schepersade}*{Proposition 5.2}, and therefore $E_{st}(X)$ is a polynomial as well. We have: \begin{equation*}\begin{split}a_{2,2}(X) = a_{2,2}(X_0)h^0(\PP^1) + h^{2,1}_{st}(X_0)h^{1,0}(\PP^1)  + h^{1,2}_{st}(X_0)h^{0,1}(\PP^1)+ h^{1,1}_{st}(X_0)h^{1,1}(\PP^1).\end{split}\end{equation*} As \begin{equation*}\begin{split}a_{2,2}(X_0) =& \  h^{2,2}(Y_0) - h^{1,1}(D(1))  =  \ 61 - 2\cdot 45 =  -29,\end{split}\end{equation*} and $$h^{1,1}_{st}(X_0) = h^{1,1}(Y_0) - h^0(D(1)) = 61 - 45 = 16,$$ we obtain that $$a_{2,2}(X) = -29 + 16 = -13.$$ 
Note that \conjectureref{conjecture} holds for $X$ as $$h^{2,2}_{st}(X) = a_{2,2}(X) + \sum_{a_j=1}{h^0(D_j)} = -13 + 45 = 32.$$
\end{example}

\subsection{Outline of the proof}

The proof of \theoremref{main} is contained in Sections \ref{setup} - \ref{proof}. In this section we make reference to what is being proved in each of them.

\subsubsection{}\label{mainnotation} For the rest of the paper we use the following.

\begin{notation}Let $X$ be a fourfold with at most Gorenstein terminal singularities and $C\cup S$ the singular locus. We assume that $C$ has pure dimension 1 and is not empty, and $S$ is the finite set of isolated singularities. Let $f:Y\to X$ be a log-resolution of singularities, which is an isomorphism on the smooth locus of $X$, and $D= \bigcup{D_i}$ the exceptional set. For $D_i\subseteq D$, an irreducible component of $D$, we denote $f\restr{D_i} = g_i: D_i \to C$.  For a general hyperplane section $H\subseteq X$ we denote $H'=f^{-1}(H)$. Note that $f\restr{H'} = f': H'\to H$ is a log-resolution.
\end{notation}

\subsubsection{} The first step of the proof is to get a convenient simplification of $$h^{2,2}_{st}(X) = h^{2,2}(Y) - h^{1,1}(D(1)) + h^0(D(2)) + \sum_{a_j=1}{h^0(D_j)}.$$
Applying \theoremref{dt} we get expressions (\ref{dt4}) and (\ref{dt3}), which are direct sum decompositions of $H^4(Y)$ and $H^2(D(1))$. They include the information of a stratification of $X$ and several local systems. We set the notation for these objects in Section \ref{setup}. \\

In Section \ref{cohomologies} we either give explicit computations of the dimensions of pieces of these expressions, or results showing that the dimensions of some of these in (\ref{dt4}) and (\ref{dt3}) coincide.\\

Finally, in Section \ref{hodgestructures} a description of the subspaces $H^{2,2}(Y)$ and $H^{1,1}(D(1))$ is given.\\

All of the above is put together in Section \ref{proof}. Using the previous results we obtain a simplification of $h^{2,2}_{st}(X)$, and we show $h^{2,2}_{st}(X) - h^{1,1}_{st}(H)\geq 0$. As the second term is nonnegative by \corollaryref{cor3folds}, we obtain the result.

\subsection{Set up}\label{setup} Applying \theoremref{dt} to the maps $f$ and $g_i$ gives the following data: a stratification of the maps and several local systems. It is convenient to set a clear notation for these objects.

\subsubsection{}\label{notation2}The stratification of $f$ includes one of $X$ we describe as follows:

\begin{notation}We have $$X = X_{reg}  \amalg U \amalg S \amalg S'\amalg S'',$$ such that:
\begin{enumerate}[label=(\alph{*}), ref=(\alph{*})]
\item The space $X_{reg}= X\setminus C\cup S$ is the open subvariety consisting of the smooth locus.
\item $U$ is an open set of $C$ contained in the smooth locus of $C$.
\item $S'= \{z_1, \ldots , z_r\}$ and $S''=\{d_1, \ldots, d_s \}$ are a finite set of points in $C$ such that $U \amalg S' \amalg S'' = C$, and $S= \{s_1,\ldots, s_t\}$ are the isolated singularities of $X$. \end{enumerate}
They satisfy the following conditions:
\begin{enumerate}[resume, label=(\alph{*}), ref=(\alph{*})]
\item If $C=\bigcup C_i$ is reducible, $U= \amalg U_i$ where $U_i\subseteq C_i$ is an open subset contained in the smooth locus of $C_i$. 
\item For $d_i \in S''$, $f^{-1}(d_i)$ is 3-dimensional.
\item We have $D= D'\cup D''\cup D'''$ and: \begin{renumerate}\item for $D_i$ a component of $D'$, $f(D_i) = C_j$ for some $j$, \item for $D_j\subseteq D''$, $f(D_j) \in S''$, \item $D'''$ is the disjoint union of the fibers of the isolated singularities. \end{renumerate}
\end{enumerate}
We can assume that for a component $B\subseteq D_i\cap D_j$, if $f(B) = z \in C$, then $z\in S''$ (which implies $D_i \subseteq D''$ or $D_j \subseteq D''$). Indeed, we can blow up the connected components of the double intersections of the irreducible components of $D'$ that contract to a point. An analogous argument works for triple intersections as well. \\

We denote $j:U \into C$ the inclusion map. Let $\nu:\tilde{C} = \amalg\tilde{C_i} \to C$ be the normalization and $k:U \to \tilde{C}$, so that $j = \nu\circ k$. We have a factorization of the map $g_i: D_i \to C$, as a composition of $g_i': D_i \to \tilde{C}$ and $\nu$, for $D_i\subseteq D'$. Note that the normalization $\nu$ is an isomorphism on $U$.
\end{notation}

\subsubsection{}The only intersection complex supported on $X$ in the decomposition of $\derR f_*\CC_Y[4]$ is $IC_X$, as the map $f$ is birational (cf. \cite[Proof of Theorem 2.2.3 d)]{cm05}). The other intersection complexes are supported on the strata of dimension 1 or 0, and for these dimensions the intersection complex of a local system has an easy description. 

\begin{remark}\label{rmkdim1}  For a local system $L$ on $U$, $IC_C(L) \cong j_*L[1]$. If $C$ is not smooth, the following isomorphisms are useful: $$IC_C(L) \cong \nu_*IC_{\tilde{C}}(L) \cong \nu_*k_*L[1].$$
For a local system supported on a zero-dimensional space, the intersection complex is the local system itself.  
\end{remark}

With these, if we apply \theoremref{dt} to $f$, we get the following descriptions:
$$^{\mathfrak{p}}\mathcal{H}^k(\derR f_*\CC_Y[4])[-k] \cong j_*L_{U,k}[1-k] \oplus L_{S,k}[-k] \oplus L_{S',k}[-k] \oplus L_{S'',k}[-k]$$ for $k= -2,-1,1,2$ and $$^{\mathfrak{p}}\mathcal{H}^0(\derR f_*\CC_Y[4])[0] \cong IC_X \oplus j_*L_{U,0}[1] \oplus L_{S,0}[0] \oplus L_{S',0}[0]\oplus L_{S'',0}[0],$$ where the first subindex corresponds to the stratum in which each local system is defined.\\

The cohomology $H^4(Y)$ is isomorphic to $\HH^0(X,\derR f_*\CC_Y[4])$. By applying $\HH^0$ to the expressions above, we get the following direct sum decomposition:

\begin{equation}\label{dt4}\begin{split} & H^0(C, j_*L_{U,1}) \\ H^4(Y)\cong\quad \oplus\ IH^4(X)\ \oplus\ & H^1(C, j_*L_{U,0}) \oplus H^0(S\amalg S'\amalg S'', L_{S,0}\oplus L_{S',0}\oplus L_{S'',0}) \\ \oplus \ & H^2(C,j_*L_{U,-1})
\end{split}\end{equation}
where each line of the sum corresponds to $\HH^0$ of a perverse cohomology of $\derR f_*\CC_Y[4]$.

\subsubsection{} Let $D_i \subseteq D'$ be an irreducible component. The map $g'_i$ goes form a smooth variety of dimension 3 to a smooth curve. The application of the Decomposition Theorem to these kinds of maps is discussed in \cite[Example 2.5]{cm05}: 

\begin{example}\label{excm} We can assume this map is smooth over $U\subseteq \tilde{C}$. The computation says: $$ ^{\pp}\cohH^m(\derR g'_{i*}\CC_{D_i}[3]) \cong k_*R^{2+m}g'_{i*}\CC_{U'_i}[1] \oplus K_i^m$$ where $U'_i= {g_i^{\prime -1}(U)}$. Note that by taking the cohomology of the complexes we get $$R^{m+3}g'_{i*}\CC_{D_i} \cong k_*R^{m+3}g'_{i*}\CC_{U'_i} \oplus K_i^{m},$$ which implies that if we take the stalk at $z\in g'_{i}(D_i) \setminus U$ we have the following isomorphism: $$ H^{m+3}(g_i^{\prime -1}(z)) \cong H^{m+3}(g_i^{\prime -1}(x))^{\ZZ} \oplus (K_i^m)_z,$$ where $\ZZ$ is the monodromy action around the point $z$ and $x\in g'_{i}(D_i) \cap U$ is a nearby point. The above discussion gives a description of the sheaves $K_i^m$. 
\end{example}

Using that $H^2(D_i) = \HH^{-1}(\derR g_{i*}\CC_{D_i}[3]) = \HH^{-1}(\nu_*\derR g_{i*}'\CC_{D_i}[3])$ and that $g_i = \nu \circ g'_i$, we obtain the following direct sum decomposition: 

\begin{equation}\begin{split}\label{dt3}H^2(D_i)\cong H^2(C, j_*g_{i*}\CC_{U'}) \oplus  H^1(C, j_*R^1g_{i*}\CC_{U'}) \oplus H^0(S, \nu_*K_i^{-1}) \oplus  H^0(U, R^2g_{i*}\CC_{U'}).
\end{split}\end{equation}

\subsection{Computation of the cohomologies}\label{cohomologies} In this section we discuss the summands in expressions (\ref{dt4}) and (\ref{dt3}). We also discuss the much simpler expression (\ref{dt3b}), which we obtain by applying \theoremref{dt} to the map $f'$.

\subsubsection{}\label{cohdt4} We use the notation of Section \ref{setup}.

\begin{lemma}\label{lu1} The space of global sections of $L_{U,1}$ has the following dimension: $$h^0(C, j_*L_{U,1}) = h^0(U,L_{U,1}) = h^0(D'(1)).$$
\end{lemma}

\begin{proof}
As $\dim\Supp{\cohH^0(IC_X)} < 0$, we have $$\cohH^0(\derR f_*\CC_Y[4]) = R^4f_*\CC_Y \cong j_*L_{U,1} \oplus L_{S,0} \oplus L_{S',0} \oplus L_{S'', 0},$$ and hence $$ R^4f_*\CC_Y \restr{U} \cong L_{U,1}.$$
This means that for $x\in U$,  $$H^4(f^{-1}(x)) \cong (R^4f_*\CC_Y)_x \cong (L_{U,1})_x.$$
Let $g: D \to C\amalg S$ be the restriction of $f$ to $D$. As $R^4f_*\CC_Y$ is supported on $C\amalg S$, this sheaf is isomorphic to $R^4g_*\CC_D$. The map $\CC_D \to \CC_{D(1)}$ induces a quasi-isomorphism $$\CC_D \to [0 \to \CC_{D(1)} \to \ldots \to \CC_{D(4)} \to 0]=:Q$$ because $D$ is a space with simple normal crossings (see \cite[Lemma and Definition 3.2.31]{elzeinetal}). Using the spectral sequence associated to the trivial filtration of the complex $Q$ we compute $R^4g_*\CC_D$. It is defined as:
$$E_1^{p,q} = R^qg_*\CC_{D(p+1)} \Rightarrow R^{p+q}g_*Q.$$
We look at $p+q=4$. Note that for $x\in U$, $(R^4g_*\CC_{D(2)})_x = 0$, as $\dim{f^{-1}(x)} = 1$. 
%By \ref{excm} $R^4g_*\CC_{D(1)}|_U$ is a local system. The map $$ 0 \to R^4g_*\CC_{D(1)} \to R^4g_*\CC_{D(2)}$$ which corresponds to $$ E_1^{-1,4} \to E_1^{0,4} \to E_1^{1,4}$$ that is given by restriction maps on the stalks, has to be the zero map on $U$. Otherwise, as $R^4g_*\CC_{D(2)}$ is a skyscraper sheaf, the kernel $E_2^{0,4}$ would not be a local system on $U$.
Therefore, the maps in the complex $$ 0 \to R^4g_*\CC_{D(1)} \to R^4g_*\CC_{D(2)},$$ which corresponds to $$ E_1^{-1,4} \to E_1^{0,4} \to E_1^{1,4},$$ are the zero map on $U$. As $E_1^{2,3} = 0$ by dimension reasons, %$E_2^{0,4} = E_{\infty}^{0,4}$ is a summand of $R^4g_*Q$. The latter is isomorphic to $L_{U,1}$ when restricted to $U$, and therefore 
$$R^4g_*\CC_{D(1)}\restr{U} \cong E_{\infty}^{0,4}\restr{U}.$$
%, hence this is not possible. Therefore $$R^4g_*\CC_{D(1)} \cong E_{\infty}^{0,4}.$$
The terms $E_1^{1,3}, E_1^{2,2}$ and $E_1^{3,1}$ are all zero on $U$ by a similar reasoning.
%$E_1^{1,3}, E_1^{2,2}$ and $E_1^{3,1}$ are all skyscraper sheaves and a similar argument allows us to conclude that the $E_{\infty}$ for this numbers should all be zero. \\
This means that we obtain the isomorphisms \begin{equation}\label{lu1iso}L_{U,1} \cong R^{4}g_*Q\restr{U} \cong R^4g_*\CC_{D(1)}\restr{U}.\end{equation}

%The stalks of $R^0g_{i*}\CC_{U'}$ are 1-dimensional vector spaces because the fibers are connected. 
%The Leray spectral sequence says that $H^0(U, R^0g_{i*}\CC_{U'}) \cong H^0(U')$ which is a 1-dimensional vector space. 
%Hence $R^0g_{i*}\CC_{U'} \cong \CC_U$. We also have that $j_*\CC_U \cong \CC_C$. 
%The duality of the intersection complexes showing up in the decomposition theorem tell us that $j_*R^4g_{i*}\CC_{U'} \cong j_*R^0g_{i*}\CC_{U'}$ and hence 
Suppose first $C$ is smooth and irreducible. \exampleref{excm}, implies that $$H^0(C,j_*R^4g_{i*}\CC_{U'}) = \HH^{-1}(C,\ ^{\pp}\cohH^2(\derR g_{i*}\CC_{D_i}[3])) \cong H^4_2(D_i),$$
where the last term is the perverse cohomology group with respect to $g_i$ (see \definitionref{perverse cohomology}). Using \theoremref{hlpc}\ref{hla}, we get that $$H^0_{-2}(D_i) \cong H^4_2(D_i),$$ and we also have that:
\begin{equation*}\begin{split} H^0_{-2}(D_i) = \HH^{-1}(C,\ ^{\pp}\cohH^{-2}(\derR g_{i*}\CC_{D_i}[3])) = H^0(C,j_*g_{i*}\CC_{U'})\\ = H^0(U, g_{i*}\CC_{U'}) \cong H^0(U', \CC_{U'}).
\end{split}\end{equation*}
Hence, 
 $$ H^4(U, R^4g_{i*}\CC_{U'}) \cong H^0(C,j_*R^4g_{i*}\CC_{U'}) \cong H^0(U', \CC_{U'}),$$ and $$ \dim{H^4(U, R^4g_{i*}\CC_{U'}) } = 1.$$

%Using the description that was given in each irreducible component, conclude that it is a local system with trivial monodromy. If $D$ has $l$ irreducible components, each stalk is a vector space generated by each of the $l$ irreducible components of the fiber $f^{-1}(x)$.
To conclude, since we have $$H^0(U, L_{U,1}) \cong H^0(U, R^4g_*\CC_{D(1)}\restr{U})$$ by (\ref{lu1iso}),  we obtain that $$\dim{H^0(U, L_{U,1})} = \sum{\dim{H^0(U, R^4g_*\CC_{D_i}\restr{U})}} = h^0(D'(1)).$$ 

Assume now that $C$ is an arbitrary curve.  For every $D_i\subseteq D'$, the map $g_i'$ surjects onto a smooth and irreducible curve $\tilde{C_j}$. We can use the same analysis for this map, and get the same conclusion.

\end{proof}

\begin{lemma}\label{lu-1}The following holds: $$h^2(C, j_*L_{U,-1}) = h^0(D'(1)).$$

\end{lemma}

\begin{proof}This is a simple application of \theoremref{hlpc}. We have $$H^4_1(Y) = \HH^{-1}(X,\ ^{\pp}\cohH^1(\derR f_*\CC_Y[4])) \cong H^0(U, L_{U,1}),$$ and using \theoremref{hlpc}\ref{hla} we get that $$H^2_{-1}(Y) \cong H^4_1(Y).$$ 
\theoremref{hlpc}\ref{hlpa} implies that $$H^2_{-1}(Y) \cong H^4_{-1}(Y),$$
 and finally note that $$H^4_{-1}(Y) = \HH^{1}(X,\ ^{\pp}\cohH^{-1}(\derR f_*\CC_Y[4])) \cong H^2(C, j_*L_{U,-1}).$$ 

The result follows from \lemmaref{lu1}, as $$ \dim{H^4_1(Y)}= h^0(U, L_{U,1}) = h^0(D'(1)).$$

\end{proof}

\begin{remark}\label{rmk1} In the proof we have shown that there is a chain of isomorphisms of type $(1,1)$ between these pure Hodge structures. This is used in Section \ref{hodgestructures} when we discuss the subspace $H^{2,2}(Y)$. 
\end{remark}

\begin{lemma}\label{lu0} The following terms in (\ref{dt4}) and (\ref{dt3}) have the same dimension: $$h^1(C, j_*L_{U,0}) = \sum{h^1(C, j_*R^1g_{i*}\CC_{U'})}$$

\end{lemma}

\begin{proof}Using that $R^3f_*\CC_Y = \cohH^{-1}(\derR f_*\CC_Y[4])$, and (\ref{dt4}), we have $$R^3f_*\CC_Y \cong \cohH^{-1}(IC_X) \oplus j_*L_{U,0} \oplus L_{S,-1} \oplus L_{S', -1}.$$

As $\dim\Supp{\cohH^{-1}(IC_X) < 1}$, it is supported on a finite set of points, and we can assume they are not in $U$. We use again the quasi-isomorphism $$\CC_D \to [0 \to \CC_D(1) \to \ldots \to \CC_{D(4)} \to 0]=:Q,$$ and the spectral sequence $$E_1^{p,q} = R^qg_*\CC_{D(p+1)} \Rightarrow R^{p+q}g_*Q$$ (see the proof of \lemmaref{lu1}).
We consider the terms with $p+q = 3$. For $x\in U$, $(R^3g_*\CC_{D(2)})_x=0$ by dimension reasons. 
%Recall that $R^3g_*\CC_{D(1)}|_U$ is a local system.  
The maps of the complex $$ 0 \to R^3g_*\CC_{D(1)} \to R^3g_*\CC_{D(2)},$$ which corresponds to $$ E_1^{-1,3} \to E_1^{0,3} \to E_1^{1,3},$$ are the zero map on $U$. The term $E_1^{2,2}\restr{U} = 0$ as well. We get $$R^3g_*\CC_{D(1)}\restr{U} \cong E_{\infty}^{0,3}\restr{U}.$$
%that is given by restriction maps on the stalks, has to be the zero map on $U$.  Otherwise, as $R^3g_*\CC_{D(2)}$ is a skyscraper sheaf, the kernel $E_2^{0,3}$ would not be a local system on $U$. $E_1^{2,2}$ is also a skyscraper sheaf. The kernel of the map $E_2^{0,3}\to E_2^{2,2}$ would still not be a local system. This kernel is $E_{\infty}^{0,3}$ which is a summand of $R^3g_*Q$. The latter is isomorphic to $L_{U,0}$ when restricted to $U$, hence this is not possible. Then all the previous maps have to be zero and $$R^3g_*\CC_{D(1)} \cong E_{\infty}^{0,3}.$$
Analogous arguments show $E_1^{2,1}$ and $E_1^{3,0}$ are all zero on $U$.\\

%$E_1^{2,1}$ and $E_1^{3,0}$ are skyscraper and a similar argument allows us to conclude that the $E_{\infty}$ for this numbers is zero. We are missing $$E_1^{1,2} = R^2g_*\CC_{D(2)}.$$ 
Consider the maps $$ R^2g_*\CC_{D(1)} \to R^2g_*\CC_{D(2)} \to R^2g_*\CC_{D(3)},$$ which correspond to $$ E_1^{0,2}\to E_1^{1,2} \to E_1^{2,2}.$$ For $x\in U$, the map on stalks corresponds to the alternate sum of restriction maps $$\bigoplus H^2(g_i^{-1}(x)) \stackrel{\alpha}{\to} \bigoplus H^2(g_i^{-1}(x) \cap g_j^{-1}(x)).$$ The space $f^{-1}(x)$ is a surface with simple normal crossings, the fibers $g_i^{-1}(x)$ are smooth, and are the components of $f^{-1}(x)$. Therefore, $$\ker{\alpha} \cong {\rm Gr}^W_2(H^2(f^{-1}(x)))$$ and $$\coker{\alpha} \cong {\rm Gr}^W_2(H^3(f^{-1}(x)))$$ (see (\ref{snchp})).
The surface $f^{-1}(x)$ shows up in a resolution of a threefold with isolated singularities, and hence $$\coker{\alpha} = 0$$ \cite[Corollary 1.12]{St83}. This means that $E_2^{1,2}\restr{U} = 0$.\\
%This means that $E_2^{1,2}$ is a skyscraper sheaf and we can conclude as before.\\

We obtain that $$R^3g_*\CC_{D(1)}\restr{U} \cong L_{U,0}.$$ In particular, \begin{equation}\label{R31}\dim{H^1(C, j_*R^3g_*\CC_{U'})} = \dim{H^1(C, j_*L_{U,0})}.\end{equation}
Suppose $C$ is smooth and irreducible. We have: $$H^4_1(D_i) \cong H^1(C, j_*R^3g_*\CC_{U'}) \oplus H^0(S', K_i^1)$$ and $$H^2_{-1}(D_i) \cong H^1(C, j_*R^1g_*\CC_{U'}) \oplus H^0(S', K_i^{-1}),$$ and they are isomorphic by \theoremref{hlpc}\ref{hla}. We also have $$H^0(S', K_i^1) \cong H^0(S', K_i^{-1})$$ as they are dual to each other (see for instance \cite[Section 4.4]{cm07}). Therefore, we get \begin{equation}\label{R32}\dim{H^1(C, j_*R^3g_{i*}\CC_{U'})} = \dim{H^1(C, j_*R^1g_{i*}\CC_{U'})}\end{equation}
The result follows from (\ref{R31}) and (\ref{R32}).\\

With no restriction on $C$, we use the map $g_i':D_i \to \tilde{C}$. By the previous analysis $$\dim{H^1(\tilde{C}, k_*R^3g_{i*}'\CC_{U'})} = \dim{H^1(\tilde{C}, k_*R^1g_{i*}'\CC_{U'})}.$$ Since $$\dim{H^1(\tilde{C}, k_*R^3g_{i*}'\CC_{U'})} = \dim{H^1(C, j_*R^3g_{i*}'\CC_{U'})}$$ and $$\dim{H^1(\tilde{C}, k_*R^1g_{i*}'\CC_{U'})} = \dim{H^1(C, j_*R^1g_{i*}'\CC_{U'})},$$because $\nu$ is finite, the result follows.

\end{proof}

This concludes the discussion of terms which include intersection complexes supported on the 1-dimensional strata. We examine next those with local systems supported on the 0-dimensional strata.

\begin{lemma}\label{ls0} The following dimensions of terms in (\ref{dt4}) and (\ref{dt3}) coincide: $$ h^0(S', L_{S',0}) = \sum{h^0(S', \nu_*K_i^{-1})}.$$

\end{lemma}

\begin{proof} Let $z\in S'$. As explained in \lemmaref{lu1}, $$(R^4f_*\CC_Y)_z \cong (j_*L_{U,1})_z \oplus L_{z, 0}$$ and  $$(R^4g_{i*}\CC_{D_i})_z \cong (j_*R^4g_{i*}\CC_{U'})_z \oplus (\nu_*K_i^1)_z.$$ Using the quasi-isomorphism $$\CC_D \to Q,$$ and the spectral sequence $$E_1^{p,q} = (R^qg_*\CC_{D(p+1)})_z \Rightarrow (R^{p+q}g_*Q)_z,$$ we compute $(R^4g_*Q)_z$ (see proof of \lemmaref{lu1}).\\

We look at the terms with $p+q =4$. The maps $$ 0 \to (R^4g_*\CC_{D(1)})_z \to (R^4g_*\CC_{D(2)})_z,$$ corresponding to $$E_1^{-1,4} \to E_1^{0,4} \to E_1^{1,4},$$ are the zero map. Indeed, if $(R^4g_*\CC_{D(2)})_z = H^4(g_i^{-1}(z) \cap g_j^{-1}(z)) \neq 0$, then a component of $D_i\cap D_j$ is contracted to a point. By assumption this is not possible, because $z\notin S''$. As $E_1^{2,3} = 0$ by the same reason, we see that $$E_1^{0,4}= E_{\infty}^{0,4}.$$
In the same way, $E_1^{1,3}, E_1^{2,2}$ and $E_1^{3,1}$ are all zero. Hence $$ (R^4g_*\CC_{D(1)})_z \cong (R^4g_*Q)_z\cong (R^4f_*\CC_Y)_z.$$ 

By \lemmaref{lu1} we have $$ L_{U,1} \cong \bigoplus R^4g_{i*}\CC_{U'},$$ which implies $$\dim{H^0(S', L_{S',0})} = \sum\dim{H^0(S', \nu_*K_i^{1})}$$ by adding the dimension of the stalks of every $z\in S'$. Using $$H^0(S', \nu_*K_i^{-1}) \cong H^0(S', \nu_*K_i^1)$$ the result follows.
\end{proof}

%\begin{remark}In \lemmaref{lu0} it was shown that the isomorphism $$H^4_1(D_i) \cong H^2_{-1}(D_i)$$ induced an isomorphims $$H^0(S, K_i^{-1}) \cong H^0(S, K_i^1)$$ of Hodge structures of type $(1,1)$. 
%\end{remark}

\begin{lemma}\label{ls00}For $d\in S''$ \begin{equation*} \begin{split}\dim{H^{2,2}(H^0(L_{d,0})}) & = h^{1,1}(f^{-1}(d)(1)) - h^0(f^{-1}(d)(2))\\ - &\  \sum_{\substack{ x_j \text{ near } \\ d_j\in \nu^{-1}(d)}}{(\dim{H^{2,2}(H^4(f^{-1}(x_j))^{\ZZ})})}\end{split}\end{equation*} where $f^{-1}(d)(1)$ and $f^{-1}(d)(2)$ is the usual notation for a simple normal crossings divisor (see Section \ref{d(p)}).

\end{lemma}

\begin{proof} From (\ref{dt4}) we get $$H^4(f^{-1}(d)) \cong (j_*L_{U,1})_d \oplus L_{d, 0}.$$  If $d \notin g_i(D_i)$, then $(j_*R^4g_{i*}\CC_{U'})_d = 0$. Otherwise using the isomorphism (\ref{lu1iso}) we obtain, \begin{equation*}\begin{split} (j_*L_{U,1})_d = (\nu_*k_*L_{U,1})_d = & \ \bigoplus_{d_j\in \nu^{-1}(d)} (k_*L_{U,1})_{d_j} \\ \cong & \ \bigoplus_{d_j\in \nu^{-1}(d)} H^4(f^{-1}(x_j))^{\ZZ}  \end{split}\end{equation*} where $x_j \in g_i(D_i)\cap U$ is near $d_j$, and $\ZZ$ is the monodromy around the point $d_j$.
% \begin{equation*}\begin{split} (j_*R^4g_{i*}\CC_{U'})_d = (\nu_*k_*R^4g_{i*}\CC_{U'})_d = & \ \bigoplus_{d_j\in \nu^{-1}(d)\cap g_i(D_i)} (k_*R^4g_{i*}\CC_{U'})_{d_j} \\ \cong & \ \bigoplus_{d_j\in \nu^{-1}(d)\cap g_i(D_i)} H^4(g_i^{-1}(x_j))^{\ZZ}  \end{split}\end{equation*} 
The result follows from the computation of the $H^{2,2}$ piece of the cohomology of a simple normal crossings divisor (see Section \ref{mhssnc}).

\end{proof}

\begin{lemma}\label{iso} For the local system supported on the isolated singularities of $X$, $$h^{2,2}(H^0(S, L_{S,0})) = h^{1,1}(D'''(1)) - h^0(D'''(2)).$$

\end{lemma}

\begin{proof} Let $s\in S$. In this case we have that $$H^4(f^{-1}(s)) \cong (R^4f_*\CC_Y)_s \cong L_{s,0}.$$ We know that $H^0(s, L_{s,0})$ is a pure Hodge substructure of $H^4(Y)$, and $f^{-1}(s)$ must be a simple normal crossings divisor. Therefore, \begin{equation*}\begin{split}h^{2,2}(H^0(s,L_{s,0})) = h^{2,2}(f^{-1}(s)) & = h^{2,2}(f^{-1}(s)(1)) - h^{2,2}(f^{-1}(s)(2)) \\ & = h^{1,1}(f^{-1}(s)(1)) - h^0(f^{-1}(s)(2)) \end{split} \end{equation*} (see Section \ref{mhssnc}).

\end{proof}

\subsubsection{}\label{h2h} We discuss the application of \theoremref{dt} to the map $f'$. Let $T = U\cap H = \{x_{j,l}\}$, where $x_{j,l}\in U_j\subseteq C_j$. As it is a log-resolution of singularities of a threefold with isolated singularities, we have: 
\begin{equation} \label{dt3b}\begin{split} ^{\pp}\cohH^1(\derR f'_*\CC_{H'}[3])[-1] \cong \bigoplus{H^4(f^{-1}(x_{j,l}))_{x_{j,l}}}[-1] \\
^{\pp}\cohH^0(\derR f'_*\CC_{H'}[3])[0] \cong IC_H \oplus \bigoplus{H^3(f^{-1}(x_{j,l}))_{x_{j,l}}}[0] \\
^{\pp}\cohH^{-1}(\derR f'_*\CC_{H'}[3])[1] \cong \bigoplus{H_4(f^{-1}(x_{j,l}))_{x_{j,l}}}[1].
\end{split}\end{equation}
by the computation in \cite[Remark 4.4.3]{cm07}.\\

We obtain that $$H^2(H') \cong \HH^{-1}(H, \derR f'_*\CC_{H'}[3]) \cong IH^2(H) \oplus \bigoplus{H_4(f^{-1}(x_{j,l}))}.$$
We also have that the inclusion $ H_4(f^{-1}(x_{j,l}))\subseteq H^2(H')$ is given by the composition $$H_4(f^{-1}(x_{j,l})) \into H_4(H') \cong H^2(H').$$ The space $H_4(f^{-1}(x_{j,l}))$ is generated by the classes of the irreducible components of $f^{-1}(x_{j,l})$, and its image is contained in $H^{1,1}(H')\subseteq H^2(H')$.

\subsubsection{}We discuss next the terms in expression (\ref{dt3}) which were not part of the results of Section \ref{cohdt4}.

\begin{lemma}\label{jgqu} The spaces $H^2(C, j_*g_{i*}\CC_{U'})$ are 1-dimensional.
\end{lemma} 

\begin{proof}We have $$H^2_{-2}(D_i) =\HH^1(C,\ ^{\pp}\cohH^{-2}(\derR g_{i*}\CC_{D_i}[3])) \cong H^2(C, j_*g_{i*}\CC_{U'}),$$ and \theoremref{hlpc}\ref{hlpa} implies $$ H^0_{-2}(D_i) \cong H^2_{-2}(D_i).$$ The result follows, as $\dim{H^0_{-2}(D_i)} = 1$ (see proof of \lemmaref{lu1}).
\end{proof}

Consider the Stein factorization 
\[
\begin{tikzcd}[column sep=1.5em]
D_i \arrow{rr}{g_i'} \arrow{dr}{\bar{g_i}} && \tilde{C} \\
& C' \arrow{ur}{\mu} &
\end{tikzcd}
\]
where $\mu$ is a finite map and $\bar{g_i}$ has connected fibers. The map $\bar{g_i}$ is smooth in the preimage of an open set of $C'$, and $\mu$ has no branch points in an open set of $C'$. We can assume $U$ is contained in the image of these open sets. Let $x\in U$ and $\mu^{-1}(x) = \{y_1, \ldots, y_{l_i}\}$. 

\begin{lemma}\label{r2gi}With the notation above $$H^0(U, R^2g_{i*}\CC_{U'}) \cong H^0(U, R^2g_{i*}'\CC_{U'}) \cong \im\{H^2(D_i) \to H^2(\bar{g_i}^{-1}(y_j))\}$$ for any $j$.

\end{lemma}

\begin{proof}
As the surfaces $\bar{g_i}^{-1}(y_j)$ are fibers of the smooth map $\bar{g_i}\restr{{U'}}$, they are all diffeomorphic. The map $\mu$ is finite, so $\mu_*R^2\bar{g_{i*}}\CC_{U'} \cong R^2g'_{i*}\CC_{U'}$. The Leray spectral sequence implies $$ H^0(U, R^2g'_{i*}\CC_{U'}) \cong H^0(\mu^{-1}(U), R^2\bar{g_{i*}}\CC_{U'}).$$
The second isomorphism is a consequence of applying \theoremref{gict} to $\bar{g_i}$, and the isomorphism above. The first isomorphism is a consequence of $\nu$ being a finite map. 
\end{proof}

%\begin{remark}Notice that $\bar{g_i}^{-1}(y_j)$ are the different components of $g_i^{-1}(x)$. This will be useful when we are comparing $H^2(D_i)$ to the cohomologies of the exceptional set of $f'$.
%\end{remark}

\subsection{Hodge structures}\label{hodgestructures} In this section we describe the subspaces $H^{2,2}(Y)$ and $H^{1,1}(D_i)$.   

%\subsubsection{} The Hodge structure of the pieces of the direct sum decomposition $H^2(H')$ given by \theoremref{dt} was discussed in Section \ref{h2h}. The conclusion was that $$H_4(f^{-1}(x_{j,l})) \subseteq H^{1,1}(H').$$
%We use again the description of \cite{cm07}. The skyscraper sheaves $H_4(f^{-1}(x_i))$ are identified in $H^2(H')$ via the map $$ H_4(f^{-1}(x_{j,l})) \into H_4(H') \stackrel{P.D.}{\to} H^2(H').$$ The image lands in $H^{1,1}(H')$ as it is generated by the classes of the irreducible components of these divisors.\\
%This will be useful for the other descriptions as we will see.

\subsubsection{} We describe first the $H^{2,2}$ pieces of the summands in (\ref{dt4}).

\begin{lemma} The following holds: $$H^0(U, L_{U,1})\subseteq H^{2,2}(Y) \subseteq H^4(Y)$$ and $$ H^2(C,j_*L_{U,-1}) \subseteq H^{2,2}(Y) \subseteq H^4(Y).$$
\end{lemma}

\begin{proof} Recall that $$H^2_{-1}(Y) \cong H^0(U,L_{U,-1})$$ and $$H^2_{-1}(H') \cong \bigoplus_{i=1}^k{H_4(f^{-1}(x_i))}$$ (see \definitionref{perverse cohomology}). 
The map $H^2(Y)\stackrel{v^*}{\to} H^2(H')$ respects the perverse filtration by \lemmaref{lemmahyp}. \propositionref{prophyp} implies that $$\HH^{-1}(X,\ ^{\pp}\cohH^{-1}(\derR f_*\CC_Y[4])) \to \HH^{-1}(H, u^{*\pp}\cohH^{-1}(\derR f_*\CC_{Y}[4]))$$ is an injection. By \lemmaref{lemmahyp} $$u^{*\pp}\cohH^{-1}(\derR f_*\CC_{Y}[4]) \cong\  ^{\pp}\cohH^{-1}(\derR f'_*\CC_{H'}[3])[1],$$ and so $$\HH^{-1}(H, u^{*\pp}\cohH^{-1}(\derR f_*\CC_{Y}[4])) \cong \HH^0(H,\ ^{\pp}\cohH^{-1}(\derR f'_*\CC_{H'}[3])) = H^2_{-1}(H').$$
The conclusion is that $$H^2_{-1}(Y) \to H^2_{-1}(H')$$ is an inclusion. As the restriction map is one of Hodge structures, and $$H^2_{-1}(H')\subseteq H^{1,1}(H')$$ as discussed in Section \ref{h2h}, we obtain: $$H^2_{-1}(Y)\subseteq H^{1,1}(Y).$$
As explained in Remark \ref{rmk1},  $$H^2_{-1}(Y) \cong H^4_1(Y)$$ via a $(1,1)$-map, hence $$H^4_1(Y)\subseteq H^{2,2}(Y).$$ Analogously, we get $$H^4_{-1}(Y) \subseteq H^{2,2}(Y).$$

\end{proof}

\subsubsection{} Finally we describe the $H^{1,1}$ pieces of the summands of (\ref{dt3}). 

\begin{lemma} The space $H^2(C, j_*g_{i*}\CC_{U'})$ is contained in $H^{1,1}(D_i).$
\end{lemma}

\begin{proof}In the proof of \lemmaref{jgqu} it was shown that $$H^0_{-2}(D_i) \cong H^2_{-2}(D_i)$$ via a (1,1) map. As $H^0(D_i)= H^{0,0}(D_i)$, the result follows.
\end{proof}

\subsection{Proof}\label{proof} In this section we prove \theoremref{main}. 

\subsubsection{} The theorem is proved assuming condition $(*)$. Recall the definition:\\ \begin{align*} (*) &&
\parbox[c]{13cm}{  If $D\subseteq Y$ is the exceptional set of $f$, with irreducible components $D= \bigcup D_i$, $f\restr{D_i}$ has connected fibers, and for any irreducible component $B_{ij}\subseteq D_i\cap D_j$, $f\restr{B_{ij}}$ has connected fibers.}
\end{align*}

\begin{proof}[Proof of \theoremref{main}] We show first that $$h^{2,2}_{st}(X) \geq h^{1,1}_{st}(H)$$ for a general hyperplane $H \subseteq X$. This inequality is equivalent to: $$h^{2,2}(Y) - h^{1,1}(D(1)) + h^0(D(2)) + \sum_{a_i = 1} h^0(D_i) - h^{1,1}(H') + h^0(E(1)) \geq 0$$ where $E = D\cap H'$. \\

Consider the first part of the sum: $$h^{2,2}(Y) - h^{1,1}(D(1)).$$ Combining the Lemmas in Sections \ref{cohomologies} and \ref{hodgestructures} we get 
\begin{equation*} \begin{split} 
h^{2,2}(Y) - h^{1,1}(D(1))  = \big[ 2h^0(D'(1)) +\dim{H^{2,2}(IH^4(X))} - h^0(D''(2)) - h^0(D'''(2)) \\ - \sum_{d\in S''}\sum_{\substack{ x_j \text{ near } \\ d_j\in \nu^{-1}(d)}}{(\dim{H^{2,2}(H^4(f^{-1}(x_j))^{\ZZ})})}\big] - \big[h^0(D'(1)) + \sum h^{1,1}(H^0(U, R^2g_{i*}\CC_{U'}))\big]
\end{split} \end{equation*} after canceling the terms of \lemmaref{lu0} and \lemmaref{ls0}.\\

Note that $$ h^0(f^{-1}(x)(1)) \geq \dim{H^{2,2}(H^4(f^{-1}(x))^{\ZZ})}.$$ Using this, together with the discussion of Section \ref{h2h}, we obtain the following inequality:
\begin{equation*} \begin{split} h^{2,2}_{st}(X) - h^{1,1}_{st}(H)  & \geq h^0(D'(1)) +\dim{H^{2,2}(IH^4(X))} - \dim{H^{1,1}(IH^2(H))} \\  & - h^0(D''(2)) - h^0(D'''(2)) + h^0(D(2))  + \sum_{a_i=1}{h^0(D_i)}  \\ &- \sum_{d\in S''}\sum_{\substack{ x_j \text{ near } \\ d_j\in \nu^{-1}(d)}}{h^0(f^{-1}(x_j)(1))}  - \sum h^{1,1}(H^0(U, R^2g_{i*}\CC_{U'})).
\end{split} \end{equation*}

Applying \theoremref{wlt} we get $$IH^2(X) \cong IH^2(H),$$ and \theoremref{hlt} implies there is an injection $$IH^2(X) \into IH^4(X).$$ These maps respect the Hodge structures of the intersection cohomologies, and therefore $$\dim{H^{2,2}(IH^4(X))} - \dim{H^{1,1}(IH^2(H))}\geq 0.$$ 

Next, by the definitions of $D= D'\cup D'' \cup D'''$ we have $$h^0(D(2)) = h^0(D'(2)) + h^0(D''(2)) + \sum_{\substack{D_i\subseteq D'\\ D_j\subseteq D''}}{h^0(D_i\cap D_j)} + h^0(D'''(2)).$$
As the components of $D'''$ do not intersect $D'$ nor $D''$, because they are the fibers of the isolated singularities, these terms cancel out in the right hand side of the inequality above. On the other hand, as we are assuming condition $(*)$, the general fiber of $f\restr{D_i}$ is connected, hence for each component of $D'$ corresponds one component of $f^{-1}(x)$. Also, every $D_i\subseteq D'$ must intersect a component of $f^{-1}(d)$, for all $d\in S''$. This means that $$ \sum_{\substack{D_i\subseteq D'\\ D_j\subseteq D''}}{h^0(D_i\cap D_j)} \geq \sum_{d\in S''}\sum_{\substack{ x_j \text{ near } \\ d_j\in \nu^{-1}(d)}}{h^0(f^{-1}(x_j)(1))}.$$ Note that if $\nu^{-1}(d) \cap \tilde{C_j}$ has $k$ points, then for every $D_i$ such that $g_i(D_i)= C_j$, $D_i \cap f^{-1}(d)$ has at least $k$ components. Indeed, $g_i^{-1}(d) = D_i \cap f^{-1}(d)$, and $g_i = \nu \circ g_i'$.\\ 

With the discussion above we can simplify the inequality: \begin{equation*} \begin{split} 
h^{2,2}_{st}(X) - h^{1,1}_{st}(H) & \geq h^0(D'(1)) + h^0(D'(2)) \\ &- \sum h^{1,1}(H^0(U, R^2g_{i*}\CC_{U'})) + \sum_{a_i=1}{h^0(D_i)}.
\end{split} \end{equation*}
We also have $$ h^{1,1}(g_i^{-1}(x)) \geq h^{1,1}(H^0(U, R^2g_{i*}\CC_{U'})), $$ that comes from assuming the local system has trivial monodromy, and note that $$\sum{h^{1,1}(g_i^{-1}(x))} = h^{1,1}(f^{-1}(x)(1)).$$

We obtain \begin{equation*} \begin{split} 
h^{2,2}_{st}(X) - h^{1,1}_{st}(H)  \geq h^0(D'(1)) + h^0(D'(2)) -  \sum h^{1,1}(f^{-1}(x_k)(1))+ \sum_{a_i=1}{h^0(D_i)}\end{split} \end{equation*} where for every connected component of $U$, say $U_k$, we pick $x_k\in U_k$. \\

Consider the right hand side of the inequality. Recall $T = U\cap H = \{x_{j,l}\}$ with $x_{j,l}\in U_j \subseteq C_j$. We want to compare it to a piece of the sum  \begin{equation*}\begin{split}h^{2,2}_{st}(H) - h^{1,1}_{st}(H) = -h^{1,1}(E(1)) + h^0(E(2)) + \sum_{a_{kli}=1}{h^0(E_{kli})} + h^0(E(1)). \end{split}\end{equation*}
Recall that $f^{-1}(x_{k,l}) = E_{kl} = E_{kl1} \cup \cdots \cup E_{klm}$. If we fix $k$ and vary $l$, the fibers do not intersect, they all have the same number of components, double intersections, and components with discrepancy one. This means that the sum similar to the one right hand side above, but with $E_{kl}$ instead of $E$, is the same for all $l$. \\

We fix $l$ from now on. As to each $D_i\subseteq D'$ such that $g_i(D_i) = C_k$ corresponds one component of $E_{kl}$, namely $E_{kli}$, we have $$h^0(D'(1)) = \sum_k h^0(E_{kl}(1)).$$ Also, the corresponding discrepancies are the same, hence $$ \sum_{a_{i}=1}{h^0(D_{i})} \geq \sum_{a_{kli}=1}{h^0(E_{kli})}, $$ as some components of $D''$ might have discrepancy 1. It is also clear that $$h^{1,1}(f^{-1}(x_k)(1))=  h^{1,1}(f^{-1}(x_{k,l})(1)) = h^{1,1}(E_{kl}(1)).$$
Finally, as we are also assuming the general fibers of the maps $g_i\restr{D_i\cap D_j}$ are connected in each of the components of $D_i\cap D_j$ in condition $(*)$, we have $$h^0(D'(2)) = \sum_kh^0(E_{kl}(2)).$$ 

We obtain:
\begin{equation*}\begin{split} h^{2,2}_{st}(H) - h^{1,1}_{st}(H)  & = h^0(E(1))  + h^0(E(2)) -h^{1,1}(E(1)) + \sum_{a_{kli}=1}{h^0(E_{kli})} \\ & = \sum_{k,l}{(h^0(E_{kl}(1)) + h^0(E_{kl}(2)) - h^{1,1}(E_{kl}(1)) +\sum_{a_{kli}=1}{h^0(E_{kli})} )}
\end{split}\end{equation*} 

The conclusion of the discussion above is that 
\begin{equation*}\begin{split}h^0(D'(1)) + h^0(D'(2)) -  \sum h^{1,1}(f^{-1}(x_k)(1))+ \sum_{a_i=1}{h^0(D_i)} \\ \geq \sum_k{(h^0(E_{kl}(1)) + h^0(E_{kl}(2)) - h^{1,1}(E_{kl}(1)) +\sum_{a_{kli}=1}{h^0(E_{kli})})}. \end{split} \end{equation*} 
The right hand side is a piece of $h^{2,2}_{st}(H) - h^{1,1}_{st}(H)$, which consists of picking one fiber per curve, that is, fixing a value of $l$. In \propositionref{h22st} it was shown that this number is nonnegative, and therefore $$h^{2,2}_{st}(X) - h^{1,1}_{st}(H) \geq 0.$$ Applying \corollaryref{cor3folds} to $H$ we obtain $$h^{2,2}_{st}(X) \geq 0.$$

\end{proof}

\subsection{Applications}\label{examples}In this section we discuss classes of fourfolds $X$ for which $h^{2,2}_{st}(X)$ is nonnegative.

\subsubsection{}\label{easyexamples} We discuss first the example of a product. 
\begin{example}Let $X_0$ be a threefold with at most terminal Gorenstein singularities and $C$ a smooth curve. Let $Y_0 \to X_0$ be a log-resolution with exceptional set $E\subseteq Y_0$ and let $X=X_0 \times C$. In this case $Y_0\times C \to X$ is a log-resolution with exceptional set $E\times C$. We have the following: $$h^{2,2}_{st}(X) = h^{2,2}_{st}(X_0) + h^{2,1}_{st}(X_0)h^{0,1}(C) + h^{1,2}_{st}(X_0)h^{1,0}(C) + h^{1,1}_{st}(X_0).$$ If $E_{st}(X_0)$ is a polynomial, then $h^{2,2}_{st}(X_0) = h^{1,1}_{st}(X_0) \geq 0$ by \theoremref{hp1}. This implies that $h^{2,2}_{st}(X) \geq 0$ by applying \theoremref{hp1} to each of the stringy Hodge numbers in the expression. But even if it is not a polynomial we can get the same conclusion by applying \corollaryref{cor3folds}, since it implies $$h^{2,2}_{st}(X_0) \geq 0,$$ we obtain as above $$h^{2,2}_{st}(X) \geq 0.$$\end{example}
 
\subsubsection{}\label{hardexamples} We describe next a class of terminal fourfolds $X$ that admit a log-resolution which satisfies condition $(*)$.\\

For the purpose of this paper we introduce the following terminology. 

\begin{definition}\label{equisingular} Let $X$ be a fourfold with Gorenstein terminal singularities. Let $C\subseteq X$ be an irreducible curve in the singular locus of $X$. We say that $X$ is \emph{equisingular} along $C$ if it is locally a hypersurface in $\CC^5$, say with coordinates $x,y,z,w,t$, and $C$ is locally a complete intersection given by $(x=y=z=w=0)$, such that, outside a finite set of points, the singularities in the hyperplanes $t=a\in \CC$ are analytically isomorphic. We say that $X$ is \emph{strongly equisingular} along $C$ if it is equisingular and the condition is satisfied on every point of $C$.\end{definition} 

\begin{definition}\label{controlledresolution} We say that a singularity of a terminal threefold admits a \emph{controlled resolution} if it admits a log-resolution of singularities such that: \begin{renumerate} \item it consists of a sequence of blow-ups of points, such that at each step the exceptional divisor has only isolated cDV singularities\footnote{A better approach would perhaps be, instead of taking a sequence of regular blow ups, to take explicit resolutions in the sense of Chen, which consist of a sequence of weighted blow ups \cite{chen16}.} which do not have the same cDV type or do not have the same Milnor number, \item\label{it2} the double intersections of the irreducible components of the exceptional divisor are connected.\end{renumerate} We say that it admits a \emph{strong controlled resolution} if admits a controlled one, such that after each blow-up the exceptional divisor has a unique singular point.  
\end{definition}

\begin{proposition}\label{propequi} Let $X$ be a fourfold with Gorenstein terminal singularities and let $C$ be the one-dimensional singular locus. Suppose that one of the following is true: \begin{renumerate}\item\label{item1} For every irreducible component of $C$, $X$ is equisingular along it and the singularity on the hyperplane sections admits a strong controlled resolution. \item\label{item2} Every connected component of $C$ is irreducible, $X$ is strongly equisingular along it, and the singularity on the hyperplane sections admits a controlled resolution.\end{renumerate} Then $X$ admits a log-resolution which satisfies condition $(*)$.  \end{proposition}

\begin{proof}We can restrict to one of the irreducible components. For \ref{item1}, we can consider the intersection points with other components as points on which $X$ is not equisingular along $C$, as this does not affect the fibers. From now on we assume $C$ is irreducible.\\

Let $X_0$ be an analytic threefold with an isolated singularity of the type determined by $C$, and $$Y_0 = X_0^n \to X_0^{n-1} \to \cdots \to X_0^1 \to X_0$$ be the sequence of blowups of the (strong) controlled resolution. We can assume $X_0^1 \to X_0$ is the blowup of the point. Let $$X^1\to X$$ be the blow up of $X$ along $C$. Let $U\subseteq X$ be one open set which makes $X$ satisfy the condition of equisingularity (see \definitionref{equisingular}). In this open set, we are taking the blow up of the ideal $(x,y,z,w)$ and we denote the blow up by $U^1$. Let $U_k$ be the subvarieties in $U$ defined by $t=k\in \CC$. In each of the $U_k$ we are taking the blow up along the only singular point, denoted by $U_k^1$. As $U_k\cong X_0$, the first blowup is isomorphic to $X_0^1$. The next step $X^2_0\to X_0^1$ is a blow up along a subvariety of $X_0^1$, which determines a subvariety on $U_k^1$. Taking the same ideal on $U^1$, gives a subvariety such that when restricted to $U^1_k$ was the original one. We define $$U^2\to U^1$$ to be the blow up along that subvariety.\\

The blow up is well defined. Indeed, let $V\subseteq X$ be a different open set such that $U\cap V \neq \emptyset$, and $V^1\to V$ the first blowup. For \ref{item1}, as the resolution of $X_0$ is strongly controlled, the subvariety determined by $V^1$ and $U^1$ on $V^1\cap U^1$ agree, as is just the new singular locus on each hyperplane section. For \ref{item2}, each of the isolated singularities in the $U^1_k$ must coincide with the corresponding singularities in the subvarieties of $V^1$. Indeed, on an analytic open set the cDV type of the singularity is the same (see \cite{namikawa01}), and outside of finite points, the Milnor number is constant in an open set around the point (see \cite[Proposition 2.57]{gls}). The strongly equisingular condition ensures that the subvariety we blow up in the next step has different connected components corresponding to different singularities we get in each fiber, and in particular we get a global subvariety that is being blown up.\\

By using the same argument as above on every step of the (strong) controlled resolution, we get a birational morphism $$Y' = X^n \to X^{n-1} \to \cdots \to X^1\to X$$ such that $Y'$ only has singularities whose image in $X$ is a finite set of points. Let $Y\to Y'$ be a log-resolution of singularities which is an isomorphism outside of the singular locus of $Y'$. To check that the conditions of $(*)$ are satisfied on $Y$, is enough to check them on $Y'$. The exceptional divisors in $Y'$ come from the irreducible components being blown up during the process. Therefore, to each of these corresponds one of the divisors in $Y_0$. The double intersections in $Y'$ correspond to double intersections in $Y_0$ by the condition \ref{it2} of controlled resolutions. Therefore the resolution satisfies $(*)$. 
\end{proof}

\begin{example}For certain types of singularities of terminal threefolds explicit resolutions have been constructed. For instance, it can be checked that singularities of type $A_n$ and $E_6$ admit strong controlled resolutions and those of type $D_{2k+1}$ admit controlled resolutions (see \cite{daisroczen01}).\end{example}

There are also some examples coming from the MMP of varieties for which \propositionref{propequi} can be applied. In \cite{tak99} and \cite{aw98} the authors study the singularities of an extremal contraction of a smooth fourfold. Let $Z$ be a smooth fourfold and $h: Z \to X$ a contraction of a extremal ray. Assume that the contraction is of type $(3,1)$, that is, the exceptional set is a divisor and is contracted to dimension 1. A complete classification is given: let $E$ be the exceptional divisor of $h$ and consider $h\restr{E}: E \to C$. This map is either a $\PP^2$-bundle or a quadric bundle. In the case of a $\PP^2$-bundle the variety is smooth or the singularities are not of index 1. In the other cases $X$ have as singular locus the smooth curve $C$. In the list of possible local equations, we can verify that they satisfy the equisingularity condition involving the singularities $A_1$ and $A_2$. Note that in the case of $A_2$, the map $h$ is not a log-resolution.

\bibliography{bib}

\end{document}